\documentclass[12pt]{amsart}

\usepackage{amssymb, amsmath,xypic}
\usepackage{hyperref, xcolor}
\usepackage[margin=1in,includefoot]{geometry}
\usepackage{paralist, comment}

\newtheorem{theorem}{Theorem}[section]

\newtheorem{proposition}[theorem]{Proposition}

\newtheorem*{recap}{Proposition~\ref{character values prop}}

\theoremstyle{definition}
\newtheorem{definition}[theorem]{Definition}
\newtheorem{example}[theorem]{Example}

\newcommand{\F}{\mathbb{F}}
\newcommand{\Fq}{\F_q}
\def\FF{\F}

\newcommand{\GL}{\mathrm{GL}}

\newcommand{\CC}{\mathbb{C}}
\newcommand{\RR}{\mathbb{R}}
\newcommand{\QQ}{\mathbb{Q}}
\newcommand{\ZZ}{\mathbb{Z}}
\newcommand{\NN}{\mathbb{N}}

\newcommand{\normchi}{\widetilde{\chi}}
\newcommand{\codim}{\operatorname{codim}}
\newcommand{\symm}{\mathfrak{S}}
\newcommand{\Cusp}{\textrm{Cusp}}
\renewcommand{\P}{\mathcal{P}}
\newcommand\Irr{\operatorname{Irr}}

\newcommand\Tr{\operatorname{Tr}}
\newcommand\hook[2]{{\left( #2-#1, 1^{#1} \right)}}
\newcommand\llambda{{\underline{\lambda}}}

\newcommand\one{{\bf 1}}
\newcommand\defn[1]{{\it{#1}}}

\newcommand\ttt{\mathbf{t}}
\newcommand\hcprod{*}
\newcommand\Class{\operatorname{Class}}

\newcommand\wt{\operatorname{wt}}

\theoremstyle{remark}
\newtheorem{remark}[theorem]{Remark}
\newtheorem{question}[theorem]{Question}

\numberwithin{equation}{section}
\newcommand\qbin[3]{\left[\begin{matrix} #1 \\ #2 \end{matrix} \right]_{#3}}

\def\SS{\mathfrak{S}}

\begin{document}

\title{Absolute order in general linear groups}
\author{Jia Huang}
\address{J. Huang, Department of Mathematics and Statistics, University of Nebraska at Kearney, Kearney, NE 68849, USA}
\author{Joel Brewster Lewis}
\address{J.\ B. Lewis and V. Reiner, School of Mathematics, University of Minnesota, Minneapolis, MN 55455, USA}
\author{Victor Reiner}

\thanks{Work partially supported by NSF grants DMS-100193, DMS-1148634, and DMS-1401792.
}


\dedicatory{To the memory of J.\ A. Green, R. Steinberg, and A. Zelevinsky}

\date{\today}

\begin{abstract}
This paper studies a partial order on the general linear group $\GL(V)$ called the absolute order, derived from viewing $\GL(V)$ as a group generated by reflections, that is, elements whose fixed space has codimension one.  The absolute order on $\GL(V)$ is shown to have two equivalent descriptions, one via additivity of length for factorizations into reflections, the other via additivity of fixed space codimensions.   Other general properties of the order are derived, including self-duality of its intervals.

Working over a finite
field $\FF_q$, it is shown via a complex character computation 
that the poset interval from the identity to a Singer cycle (or any regular elliptic element) in $\GL_n(\FF_q)$ has a strikingly simple formula for the number of chains passing through a prescribed set of ranks. 
\end{abstract}

\maketitle


\section{Introduction}

This paper studies, as a reflection group, the full general linear group $\GL(V)\cong \GL_n(\FF)$, where 
$V$ is an $n$-dimensional
vector space over a field $\FF$.  An element $g$ in $\GL(V)$
is called a \defn{reflection} if its \defn{fixed subspace} 
$V^g:=\{v\in V:\ gv=v\} = \ker(g - 1)$ has codimension $1$.
A \defn{reflection group} is a subgroup of $\GL(V)$ generated by reflections.\footnote{Our definitions here deviates slightly from the literature, where one often insists that a reflection have \emph{finite order}.  In particular, by our definition, the determinant of a reflection $g$ need not be a root of unity in $\FF^\times$, and $\GL(V)=\GL_n(\FF)$ is still generated by reflections even when $\FF$ is infinite.}  It is not hard to show
that $\GL(V)$ itself is generated by its subset $T$ of reflections, and hence is a reflection group.

{\it Finite, real} reflection groups $W$ inside $\GL_n(\RR)\cong\GL(V)$ are well-studied classically via their {\it Coxeter presentations} $(W,S)$.  Here $S$ is a choice of $n$ generating {\it simple reflections}, which are the orthogonal 
reflections across hyperplanes that bound a fixed choice of {\it Weyl chamber} for $W$.
Recent work by Brady and Watt \cite{BradyWatt2} and Bessis \cite{Bessis} has focused attention on an alternate presentation, generating real reflection groups $W$ by their subset $T$ of {\it all} reflections.  Their work makes use of the coincidence, first proven by Carter \cite{Carter},
between two natural functions  $W \rightarrow \{0,1,2,\ldots\}$ defined as follows for $w \in W$:
\begin{compactitem}
\item
the {\it reflection length}\footnote{Warning:  
this is {\it not} the usual Coxeter group length $\ell_S(w)$ coming from the Coxeter system $(W,S)$.} 
given by
$
\ell_T(w):=\min\{\ell: w=t_1 t_2 \cdots t_\ell \text{ with }t_i \in T\},
$
and 
\item the {\it fixed space codimension} 
given by
$
\codim(V^w):= n-\dim(V^w).
$
\end{compactitem}
While both of these functions can be defined for all reflection groups, it has been observed (see, e.g., Foster-Greenwood \cite{FosterGreenwood}) that for non-real reflection groups, and even for most finite complex reflection groups,
these two functions differ. This leads to two partial orders, 
\begin{compactitem}
\item
the {\it $T$-prefix order}: 
$g \leq h$ if
$\ell_T(g)+\ell_T(g^{-1}h) =\ell_T(h)$, and
\item
the {\it fixed space codimension order}: 
$g \leq h$ if
$\codim(V^g)+\codim(V^{g^{-1}h})=\codim(V^h)$.
\end{compactitem}
We discuss some general properties of these orders in Section~\ref{section:generalities}.
One of our first results, Proposition~\ref{length=codim}, is the observation that
when considering  \emph{as a reflection group} the full general linear group $\GL(V)$ for a finite-dimensional vector space $V$ over a field, 
one again has the coincidence $\ell_T(g)=\codim(V^g)$, 
and hence the two partial orders above give the same order on $\GL(V)$, which we call the {\it absolute order}.

We proceed to prove two basic 
enumerative results about this absolute order on $\GL(V)$ when 
the field $\FF=\FF_q$ is finite.
First, Section~\ref{ranks-in-whole-order} uses M\"obius inversion
to count the elements in $\GL_n(\FF_q)$ of a fixed poset rank,  
that is, those with $\ell_T(g)=\codim(V^g)$ fixed. 

Second, in Section~\ref{section:chains}, we examine the interval $[e,c]$ in the absolute order on $\GL_n(\FF_q)$ from the identity element $e$ to a \emph{Singer cycle} $c$.
There has been established in recent years a close analogy between
 the Singer cycles in $\GL_n(\FF_q)$ and {\it Coxeter elements} in real reflection groups; see  \cite[\S8--\S9]{CSP}, \cite[\S7]{RStantonWebb}, \cite{LRS}.  The interval from the identity to a Coxeter element in the absolute order on a real reflection group $W$ is a very important and well-behaved poset, called the poset of {\it noncrossing partitions} for $W$. 
Our main result, Theorem~\ref{thm:flag-f-vector}, gives a strikingly simple formula for the {\it flag $f$-vector} of $[e,c]$ in $\GL_n(\FF_q)$:  fixing an ordered composition $\alpha=(\alpha_1,\ldots,\alpha_{m})$ of $n=\sum_i \alpha_i$,  the number of chains
$e=g_0 < g_1 < \cdots <g_{m-1} < g_m=c$ in absolute order having $\ell_T(g_i)-\ell_T(g_{i-1})=\alpha_i$ is
\[
q^{\varepsilon(\alpha)} \cdot (q^n-1)^{m-1}
  \quad \text{ where } \quad 
    \varepsilon(\alpha):=\sum_{i = 1}^m (\alpha_i-1)(n-\alpha_i).
\]
The analogous flag $f$-vector formulas in real reflection groups are not as simple.

The proof of Theorem~\ref{thm:flag-f-vector} is involved, using a character-theoretic enumeration technique due to Frobenius, along with information about the complex characters of $\GL_n(\FF_q)$ that goes back to Green \cite{Green} and Steinberg \cite{Steinberg}.  The proof has the virtue of applying not only to Singer cycles in $\GL_n(\FF_q)$, but also to elements which are {\it regular elliptic}; see Section~\ref{section:chains} for the definition.   
Section~\ref{reformulation-section} reformulates the flag $f$-vector in terms of certain subspace arrangements.  We hope that this may lead to a more direct approach to Theorem~\ref{thm:flag-f-vector}.
Section~\ref{remarks}  collects further questions and remarks.


\subsection*{Acknowledgements}
The authors thank Christos Athanasiadis, Valentin F\'eray, Alejandro Morales, Kyle Petersen, and Dennis Stanton for helpful remarks, suggestions, and references.


\section{Length and prefix order}
\label{section:generalities}

The next few subsections collect some easy general properties of the length function with respect to a choice of generators for a group, and the resulting partial order defined in terms of prefixes of reduced expressions.  We borrow heavily from work of Armstrong \cite[\S 2.4]{Armstrong}, Bessis \cite[\S 0.4]{Bessis}, Brady and Watt \cite{BradyWatt}, and Foster-Greenwood \cite{FosterGreenwood}, while attempting to clarify the hypotheses responsible for various properties.

\subsection{Generated groups}

\begin{definition}
A {\it generated group} is a pair $(G,T)$ where $G$ is a group and $T \subseteq G$ a subset
that generates $G$ as a monoid:  every $g$ in $G$ has at least one
{\it $T$-word for $g$}, meaning a sequence $(t_1,t_2,\ldots,t_\ell)$ with $g=t_1 t_2 \cdots t_\ell$.
The length function $\ell=\ell_T: G \rightarrow \NN$ is defined by
\[
\ell(g):=\min\{\ell: g=t_1 t_2 \cdots t_\ell \text{ with }t_i \in T\}.
\]
That is, $\ell(g)$ is the minimum length of a $T$-word for $g$.  Words for $g$ achieving this minimum length 
are called {\it $T$-reduced}.
Equivalently, $\ell(g)$ is the length of the shortest directed path from the identity $e$ to $g$ 
in the Cayley graph of $(G, T)$.
\end{definition}

It should be clear from this definition that $\ell$ is {\it subadditive}, meaning that
\begin{equation}
\label{length-subadditivity}
\ell(gh) \leq \ell(g) +\ell(h).
\end{equation}
Understanding the case where equality occurs in \eqref{length-subadditivity}
motivates the next definition.

\begin{definition}[Prefix order]
\label{order-definition}
Given a generated group $(G,T)$, define a binary relation $g \leq h$ on $G$ by
any of the following three equivalent conditions.
\begin{compactenum}[(i)]
\item Any $T$-reduced word $(t_1,\ldots,t_{\ell(g)})$ 
for $g$ extends to a $T$-reduced word $(t_1,\ldots,t_{\ell(h)})$ for $h$.
\item There is a shortest directed path $e$ to $h$ in the Cayley graph for $(G,T)$ going via $g$.
\item $\ell(g)+\ell(g^{-1}h)=\ell(h).$
\end{compactenum}
\end{definition}

\noindent
Condition (i) makes the following proposition a straightforward exercise, left to the reader.

\begin{proposition}
For $(G, T)$ a generated group, the binary relation $\leq$ is a partial order on $G$, with 
the identity $e$ as minimum element.  It is graded by the function $\ell(-)$, in the sense that for any $g < h$, one has
$\ell(h)=\ell(g)+1$ if and only if there is no $g'$ with $g < g' < h.$
\end{proposition}

\begin{example}
Taking $G=\GL_2(\FF_2)$ and $T$ the set of all reflections in $G$, the Hasse diagram for $\leq$ on $G$ is as follows:
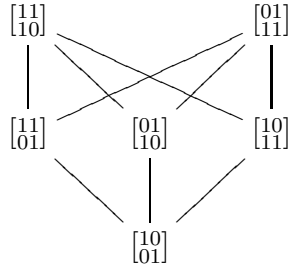
\[
\xymatrix{
\left[\substack{11 \\ 10}\right]& &\left[\substack{01 \\ 11}\right]\\
\left[\substack{11 \\ 01}\right] \ar@{-}[u] \ar@{-}[urr] &\left[\substack{01 \\ 10}\right] \ar@{-}[ul] \ar@{-}[ur] &\left[\substack{10 \\ 11}\right] \ar@{-}[u] \ar@{-}[ull]\\
 &\left[\substack{10 \\ 01}\right]\ar@{-}[ul] \ar@{-}[u] \ar@{-}[ur] &
}
\]
Coincidentally, this is isomorphic to the absolute order on the symmetric group
$\symm_3$, since the irreducible reflection representation for 
$\symm_3$ over $\FF_2$ is isomorphic to $\GL_2(\FF_2)$.
\end{example}

\subsection{Conjugacy-closed generators}

When $(G,T)$ is a generated group in which $T$ is closed under
conjugation by elements of $G$, one has  
$
\ell(ghg^{-1})=\ell(h)
$
for all $g, h$ in $G$.  This implies, for example, that
$
\ell(gh)=\ell(g^{-1} \cdot gh \cdot g)=\ell(hg).
$

The next proposition asserts an interesting consequence for the order $\leq$ on $G$,
namely that it is \defn{locally self-dual}: each interval is isomorphic to its own opposite as a poset.

\begin{proposition}
\label{prop:self-duality}
Let $(G, T)$ be a generated group, with $T$ closed under $G$-conjugacy.
Then for any $x \leq z$, the bijection $G \rightarrow G$ defined by $y \mapsto x y^{-1} z$ restricts to a poset anti-automorphism $[x, z] \rightarrow [x,z]$.  
\end{proposition}

\begin{proof}
We first check the bijection restricts to $[x,z]$.
By definition, $y \in [x,z]$ if and only if
\begin{equation}\label{y1}
\left\{
\begin{aligned}
\ell(y) &= \ell(x) + \ell(x^{-1}y),\\
\ell(z) &= \ell(y) + \ell(y^{-1}z),
\end{aligned}
\right.
\end{equation}
while $xy^{-1}z \in [x,z]$ if and only if
\begin{equation}
\label{y2}
\left\{
\begin{aligned}
\ell(xy^{-1}z) &= \ell(x)+\ell(y^{-1}z), \\
\ell(z)       &= \ell(xy^{-1}z)+\ell(z^{-1}yx^{-1}z) = \ell(xy^{-1}z)+\ell(yx^{-1}),
\end{aligned}
\right.
\end{equation}
where the last equality in \eqref{y2} uses the conjugacy hypothesis.

To see that \eqref{y1} implies \eqref{y2}, note that, assuming \eqref{y1}, one has
\begin{align*}
\ell(z)&\le  \ell(yx^{-1}) + \ell(xy^{-1}z) \\
&\le \ell(x^{-1}y) + \ell(x)+\ell(y^{-1}z) \\
&= (\ell(y)-\ell(x)) + \ell(x)+(\ell(z)-\ell(y)) 
= \ell(z),
\end{align*}
using the conjugacy hypothesis to say $\ell(yx^{-1})=\ell(x^{-1}y)$.
The fact that one has equality at each inequality above implies \eqref{y2}. Conversely, assuming \eqref{y2}, one has
\begin{align*}
\ell(z) &= \ell(x)+\ell(y^{-1}z)+\ell(yx^{-1})\\
&\ge \ell(x)+(\ell(z)-\ell(y))+(\ell(y)-\ell(x))
= \ell(z)
\end{align*}
with equality at the inequality implying \eqref{y1}.

It remains to show the restricted bijection $[x,z] \rightarrow [x,z]$ 
reverses order.  Assume $y_1\le y_2$ in $[x,z]$.  
The preceding calculations show that $\ell(xy_i^{-1}z) = \ell(x) - \ell(y_i) + \ell(z)$.  Thus
\begin{align*}
\ell(xy_1^{-1}z) & = \ell(x) - \ell(y_1) + \ell(z) \\
                 & = (\ell(x) - \ell(y_2) + \ell(z)) + (\ell(y_2) - \ell(y_1)) \\
                 & = \ell(xy_2^{-1}z) + \ell(y_1^{-1}y_2) \\
                 & = \ell(xy_2^{-1}z) + \ell(z^{-1} y_2 x^{-1} x y_1^{-1} z),
\end{align*}
using the conjugacy hypothesis in this last equality.  Hence $xy_2^{-1}z \leq xy_1^{-1}z$, as desired.
\end{proof}

The following is another 
important feature of $G$-conjugacy-closed generators $T$.
Given $g,h$ in $G$, let $g^h:=h^{-1}gh$ and ${}^hg:=hgh^{-1}$,
and note that 
\begin{equation}
\label{basis-for-Hurwitz}
g \cdot h = h \cdot g^h = {}^gh \cdot g.
\end{equation}

\begin{definition}[Hurwitz operators]
Given a generated group $(G,T)$ with $T$ closed under $G$-conjugacy
and any $T$-word 
\[
\ttt:=(t_1,\ldots,t_{i-1},t_i,t_{i+1},t_{i+2},\ldots, t_m)
\]
for $g=t_1 \cdots t_m$,
for $1 \leq i \leq m-1$ 
define the \defn{Hurwitz operator} $\sigma_i$ and its inverse $\sigma_i^{-1}$ by
\begin{align*} 
\sigma_i(\ttt)       &:=(t_1,\ldots,t_{i-1}, t_{i+1}, t_i^{t_{i+1}}, t_{i+2},\ldots, t_m), \\
\sigma_i^{-1}(\ttt) &:=(t_1,\ldots,t_{i-1}, {}^{t_i}t_{i+1},t_i, t_{i+2},\ldots, t_m).
\end{align*}
Equation \eqref{basis-for-Hurwitz} shows that $\sigma_i(\ttt)$ and $\sigma_i^{-1}(\ttt)$ are both $T$-words for $g$.
\end{definition}

\begin{remark}
Although it is not needed in the sequel, note that $\{\sigma_1,\ldots,\sigma_{m-1}\}$ satisfy the {\it braid relations} 
$\sigma_i \sigma_{i+1} \sigma_i = \sigma_{i+1} \sigma_i \sigma_{i+1}$
and 
$\sigma_i \sigma_j = \sigma_j \sigma_i
$
for $|i-j|\geq 2$, 
defining an action of the {\it braid group} $B_m$ on $m$ strands on
the set of all length-$m$ factorizations of $g$.
\end{remark}

Note that the operator $\sigma_i$ (resp.\ $\sigma_i^{-1}$) can be used to swap any letter in a word for $g$ one position to the left (resp.\ right) \emph{unchanged} at the expense of conjugating the letter with which it swapped; this creates a new word for $g$ of the same length.  Armstrong calls this the \emph{shifting property} \cite[Lem.~2.5.1]{Armstrong}.  It has the following immediate consequence.

\begin{proposition}[Subword property]
\label{prop:subword}
Let $(G, T)$ be a generated group with $T$ closed under $G$-conjugacy.
Then $g \leq h$ if and only there exists a $T$-reduced word
\[
\ttt:=(t_1,t_2,\ldots,t_{\ell(h)})
\]
for $h$ containing as a subword (not necessarily a prefix, nor contiguous) a word
\[
\hat{\ttt}=(t_{i_1},t_{i_2},\ldots,t_{i_{\ell(g)}}) \text{ with }
1 \leq i_1 < \cdots < i_{\ell(g)} \leq \ell(h)
\] 
that is $T$-reduced for $g$.
\end{proposition}

\begin{proof}
The ``only if" direction is direct from condition (i) in Definition~\ref{order-definition} of $g \leq h$.  For the ``if" direction, given the $T$-reduced word $\ttt$ for $h$
containing the $T$-reduced subword $\hat{\ttt}$ for $g$, one obtains
another $T$-reduced word for $h$ having $\hat{\ttt}$ as a prefix by repeatedly 
using Hurwitz operators
to first move the letter $t_{i_1}$ leftward (unchanged)
 to the first position, then moving $t_{i_2}$ leftward (unchanged) to the second position, etc. 
\end{proof}

\subsection{Fixed space codimension and reflection groups}

Suppose that the group $G$ is given via a faithful representation,
that is, $G$ is a subgroup of $\GL_n(\FF)=\GL(V)$ where 
$V=\FF^n$ for some field $\FF$.  This gives rise to another
subadditive function $G \rightarrow \NN$, namely the fixed space codimension
\[
g \mapsto \codim(V^g)=n-\dim(V^g).
\]

\begin{proposition}
\label{codim-subadditivity-prop}
One has the subadditivity
\begin{equation}
\label{codim-subadditivity}
\codim(V^{gh})
\leq
\codim(V^g) + \codim(V^h)
\end{equation}
with equality occurring if and only if both of the following hold:
\begin{align}
\label{spanning-condition}
V^g+V^h&= V   \qquad \textrm{ and }\\
\label{intersection-condition}
V^g \cap V^h&= V^{gh}.
\end{align}

\end{proposition}
\begin{proof}
One has
\[
\dim(V^g)+\dim(V^h)
 =    \dim(V^g+V^h)+\dim(V^g \cap V^h) 
 \leq n + \dim(V^g \cap V^h)
\]
and hence
\[
\codim(V^g)+\codim(V^h) 
\geq n-\dim(V^g \cap V^h) = \codim(V^g \cap V^h),
\]
with equality if and only if 
\eqref{spanning-condition} holds.  
Also,
$V^g \cap V^h \subseteq V^{gh}$
and so
\[
\codim(V^g \cap V^h) \geq \codim(V^{gh}),
\]
with equality if and only if
\eqref{intersection-condition} holds.
Hence
\[
\codim(V^g)+\codim(V^h) \geq \codim(V^g \cap V^h) \geq \codim(V^{gh}),
\]
with equality if and only if both conditions hold.
\end{proof}

It is natural to compare $\codim(V^g)$ with
the length function $\ell(g)=\ell_T(g)$ from before.

\begin{definition}[Absolute length, absolute order]
When a subgroup $G$ of $\GL(V)$ has a subset 
$T$ generating $G$ as a monoid, so that
$(G,T)$ is a generated group, say
that  $\ell(g)=\ell_T(g)$ is an {\it absolute length function} if 
\begin{equation}
\label{length-equals-codim}
\codim(V^g) = \ell(g) \text{ for all }g\text{ in }G.
\end{equation}
In this situation, call the prefix order $\leq$ for $(G,T)$ of
Definition~\ref{order-definition} the {\it absolute order} on $G$.
\end{definition}

\begin{proposition}
\label{length-bounds-codim}
Let $(G,T)$ be a generated group with $G$ a subgroup of $\GL(V)$.
\begin{compactenum}[(i)]
\item
If $\ell(g)$ is an absolute length function, 
then $G$ must be a reflection group and $T$ must be the set 
of all reflections in $G$.
\item
Conversely, if $G$ is a reflection group and $T$ its set of all reflections,
one at least has 
\[
\codim(V^g) \leq \ell(g) \text{ for all } g \text{ in } G.
\]
\end{compactenum}
\end{proposition}
\begin{proof}
Assertion (i) follows as $\codim(V^g)=1$ if and only
if $g$ is a reflection, and $\ell_T(g)=1$ 
if and only if $g$ lies in $T$.
For (ii), write $g=t_1 t_2 \cdots t_{\ell(g)}$ and use 
the subadditivity \eqref{codim-subadditivity}.
\end{proof}

\begin{example}
Carter showed \cite[Lem.~2]{Carter} that 
one has equality in
\eqref{length-equals-codim}
for any finite real reflection group $G \subset \GL_n(\RR)$.
\end{example}

\begin{example}
On the other hand, motivated by considerations from the theory of 
deformation of skew group rings,
Foster-Greenwood \cite{FosterGreenwood} analyzed the situation for finite \emph{complex} reflection groups $G \subset \GL_n(\CC)$ that cannot be realized as real reflection groups, and showed that
in this case it is relatively rare to have equality in 
\eqref{length-equals-codim}.

For example, the complex reflection group $G = G(4, 2, 2)$ is the set of
monomial matrices in $\CC^{2 \times 2}$ whose two 
nonzero entries lie in $\{\pm 1, \pm i\}$ and have product $\pm 1$.
It has reflections 
\[
T=
\left\{
\left[ \begin{matrix} 0 &  1 \\  1 &  0\end{matrix} \right],
\left[ \begin{matrix} 0 &  i \\ -i &  0\end{matrix} \right], 
\left[ \begin{matrix} 0 & -1 \\ -1 &  0\end{matrix} \right], 
\left[ \begin{matrix} 0 & -i \\  i &  0\end{matrix} \right], 
\left[ \begin{matrix} 1 &  0 \\  0 & -1\end{matrix} \right], 
\left[ \begin{matrix}-1 &  0 \\  0 &  1\end{matrix} \right]
\right\}
\]
and different distributions for the functions $\codim(V^g)$ and $\ell(g)$:
\[
\sum_{g \in G} t^{\codim(V^g)} = 1 + 6t + 9t^2 \qquad \textrm{ and } \qquad
\sum_{g \in G} t^{\ell_T(g)} = 1 + 6t + 7t^2 + 2t^3 .
\]
The two scalar matrices
$
\pm
\left[
\begin{smallmatrix}
i & 0\\
0 &i 
\end{smallmatrix}
\right]
$
have reflection length $3$;  neither is a product of two reflections.
\end{example}

\begin{remark}
\label{alternate-characterization-remark}
Note that whenever $G$ is a reflection group with an absolute length function,
so $\ell(g)=\codim(V^g)$, the absolute order 
relation $\leq$ acquires yet another characterization 
via Proposition~\ref{codim-subadditivity-prop}
(in addition to those in Definition~\ref{order-definition} and 
Proposition~\ref{prop:subword}).
Specifically, $g \leq h$ if and only if one has both equalities
\begin{align}
\label{codim-order-spanning-equality}
V^g+V^{g^{-1}h}&= V   \qquad \textrm{ and }     \\
\label{codim-order-intersection-equality}
V^g \cap V^{g^{-1}h}&= V^{h}.
\end{align}
\end{remark}

\begin{example}
Brady and Watt \cite{BradyWatt} considered the order
$\leq$ defined via Remark~\ref{alternate-characterization-remark}
on real orthogonal groups and complex unitary groups acting
on finite-dimensional spaces.
They showed \cite[Cor.~5]{BradyWatt} that such groups have an 
absolute length function when considered as reflection groups generated by their subset of reflections.
\end{example}

We come to our first main result, showing that the full general linear group $G=\GL(V)$ always has an absolute length function.

\begin{proposition}
\label{length=codim}
Let $G=\GL_n(\FF)=\GL(V)$ with $V=\FF^n$ for some field $\FF$,
and consider the generated group $(G,T)$ where $T$ is the set of all
reflections in $G$.
Then every $g$ in $G$ has
$$
\ell(g)=\codim(V^g).
$$
\end{proposition}
\begin{proof}
By Proposition~\ref{length-bounds-codim},
it suffices to show that $\ell(g) \leq \codim(V^g)$.
This follows by induction on $\codim(V^g)$ if
one can show that for any $g$ in $G$ other than the identity,
 there exists some $t$ in $T$ having $V^{gt} \supsetneq V^g$.
We construct such a $t$ explicitly.
Choose an ordered basis $e_1,\ldots,e_n$ for $V=W \oplus W'$ so
that $W':=V^g$ is spanned by $\{e_{m+1},e_{m+2},\ldots,e_n\}$.  In this basis for $V$, we have
\[
g=
\begin{bmatrix}
A & 0\\
B & \one_{n-m}
\end{bmatrix}
\]
where $A$ in $\GL_m(\FF)$ expresses the composite
$
W \overset{i_W}{\hookrightarrow} V \overset{g}{\to} V \overset{\pi_W}{\twoheadrightarrow} W
$
 in the basis $e_1,\ldots,e_m$. 

We claim that by making a change of basis on $W$, one may assume that 
$e_m^\top A^{-1} e_m \neq 0$.  To see this claim, fix any matrix $Q$ in $\GL_m(\FF)$
(such as $Q=A^{-1}$) having $e_m^\top Q e_m =0$.  Since $Q e_m \neq \mathbf{0}$,
there must exist some $j$ in  $\{1,2,\ldots,m-1\}$ for which $e_j^\top Q e_m \neq 0$.  Thus one may define an invertible change of basis $P$ by 
$P(e_i) = e_i$ for $i \neq j$ and $P(e_j)=e_j+e_m$.  Consequently,
$P^{-1}(e_m)=e_m$ and $P^\top e_m=e_j+e_m$, so one can calculate that $PQP^{-1}$ satisfies
\[
e_m^\top PQP^{-1} e_m 
= (P^\top e_m)^\top Q e_m 
= (e_j+e_m)^\top Q e_m 
=e_j^\top Qe_m + e_m^\top Qe_m 
=e_j^\top Qe_m \neq 0.
\]

Once one has $e_m^\top A^{-1}e_m \neq 0$,
define the desired reflection $t$ to fix the hyperplane spanned by
$\{e_1,\ldots,e_n\} \setminus \{e_m\}$ and send $e_m$ to  
$
A^{-1} e_m \oplus
(-B A^{-1} e_m)
$
in $W \oplus W'=V$.
One can check that $\det(t)=e_m^\top A^{-1}e_m \neq 0$, so that $t$ does define a reflection in $\GL(V)$.  Furthermore, both $g$ and $t$ fix $W'=V^g$ pointwise, so $gt$ also fixes $W'$ pointwise.  However, the following shows that $gt$ additionally fixes $e_m$, and hence
 $V^{gt} \supsetneq W' =V^g$, as desired:
\[
gt(e_m)=
g
\begin{bmatrix}
A^{-1} e_m \\
-B A^{-1} e_m
\end{bmatrix}
=
\begin{bmatrix}
A & 0\\
B & \one_{n-m}
\end{bmatrix}
\cdot
\begin{bmatrix}
A^{-1} e_m \\
-B A^{-1} e_m
\end{bmatrix}
=
\begin{bmatrix}
A \cdot A^{-1} e_m \\
B \cdot A^{-1}e_m - BA^{-1}e_m
\end{bmatrix}
=e_m.
\qedhere
\]
\end{proof}

\subsection{Surjection onto subspace lattices}

Consider the lattice $L(V)$ of all $\FF$-subspaces of $V=\FF^n$
ordered by \emph{reverse inclusion}.\footnote{This matches, e.g., the convention common in the theory of geometric lattices.}
For any subgroup $G$ of $\GL(V)$, one has a map 
\begin{equation}
\label{fixed-space-map}
\begin{array}{rcl}
G &\overset{\pi}{\longrightarrow}& L(V)\\
g &\longmapsto& V^g
\end{array}.
\end{equation}
If $G$ is a reflection group with an absolute length,
then Remark~\ref{alternate-characterization-remark}
shows that this map $\pi$ is order-preserving for the absolute order.  
Orlik and Solomon \cite[Lem.~4.4]{OrlikSolomon} showed that if $G$ is a
finite complex reflection group in $\GL_n(\CC)=\GL(V)$, 
then $\pi$ is a surjection onto the subposet of $L(V)$ consisting of all
subspaces that are intersections of reflection hyperplanes.
Hence for finite real reflection groups, which have an absolute
length, $\pi$ is an order-preserving surjection onto this subposet.
The next observation shows that the same holds for
the full general linear groups.  The proof is an easy exercise, left
to the reader.

\begin{proposition}
\label{surjection onto lattice of subspaces}
For $G=\GL(V)$ itself, 
the map \eqref{fixed-space-map} is an order-preserving surjection.
\end{proposition}

\begin{remark}
Brady and Watt~\cite[Thm.~1]{BradyWatt} showed that the map \eqref{fixed-space-map}
is also surjective, and in fact becomes a bijective 
order-isomorphism, when one
restricts to a lower interval $[e,c]$ between the identity $e$ and
an element $c$ having $V^c=\{0\}$ in real orthogonal or complex unitary groups.
However, this bijectivity fails for general linear groups, when typically there are many elements below $c$ having the same fixed space.  For example, it is a special case of Theorem~\ref{thm:flag-f-vector} below that there are $q^{n - 2}(q^n - 1)$ reflections in $[e, c] \subset \GL_n(\Fq)$, while there are only $(q^n - 1)/(q - 1)$ hyperplanes in $L(V)$.
\end{remark}

\begin{remark}
For finite real reflection groups, orthogonal/unitary groups, and general
linear groups, the absolute orders $\leq$ are not lattices because they have
many incomparable maximal elements.

However, when one restricts to lower intervals $[e,c]$, absolute orders 
are sometimes lattices.  For example, in the case of orthogonal/unitary groups, Brady and Watt's order-isomorphism $[e,c] \cong L(V)$ shows that every
lower interval is a lattice.  For 
irreducible finite real reflection groups in the case that $c$ is chosen to be a
{\it Coxeter element}, the fact that $[e,c]$ is a lattice
was shown originally via a case-by-case check by Bessis \cite[Fact~2.3.1]{Bessis}
and later with a uniform proof by Reading \cite[Cor.~8.6]{Reading}.

For the general linear groups $\GL(V)=\GL_n(\FF)$ with $n \geq 3$, the intervals $[e,c]$ are not lattices in general.  For example, the interval $[e, c]$ in $\GL_3(\F_3)$ below the Singer cycle
\[
c = 
\begin{bmatrix}
0 & 0 & 2 \\
1 & 0 & 1 \\
0 & 1 & 0
\end{bmatrix}
\]
contains the two reflections
\[
\begin{bmatrix}
1 & 2 & 2 \\
0 & 1 & 0 \\
0 & 1 & 2
\end{bmatrix}
\quad
\textrm{ and }
\quad
\begin{bmatrix}
1 & 2 & 2 \\
0 & 2 & 1 \\
0 & 2 & 0
\end{bmatrix},
\]
both of which are covered by three elements
\[
\begin{bmatrix}
1 & 2 & 2 \\
0 & 1 & 1 \\
0 & 1 & 0
\end{bmatrix},
\quad
\begin{bmatrix}
1 & 2 & 2 \\
1 & 0 & 1 \\
1 & 0 & 0
\end{bmatrix},
\quad 
\begin{bmatrix}
1 & 2 & 2 \\
2 & 2 & 1 \\
2 & 2 & 0
\end{bmatrix}
\]
of absolute length $2$.
\end{remark}

\subsection{Length functions when $T=T^{-1}$}

We close this section on $\ell(-)$
for a generated group $(G,T)$,
with two general facts that hold when $T=T^{-1}$, that
is, when $T$ is closed under taking inverses.
They are reminiscent of 
properties of Coxeter group length functions.

\begin{proposition}\label{prop:lengthpm}
For $(G, T)$ a generated group with $T=T^{-1}$,
any $t$ in $T$ and $g$ in $G$ have
$$
\ell(g)-1 \leq
\ell(tg), \ell(gt)
\leq \ell(g)+1.
$$ 
\end{proposition}
\begin{proof}
Subadditivity immediately gives 
$\ell(gt), \ell(tg)  \leq \ell(g)+1$.
Meanwhile
\begin{align*}
\ell(g) =\ell(gt \cdot t^{-1}) & \leq \ell(gt)+1,\\
\ell(g) =\ell(t^{-1} \cdot tg) & \leq \ell(tg)+1. \qedhere
\end{align*}
\end{proof}

\noindent
Note that $\ell(tg)=\ell(g)=\ell(gt)$ is possible, e.g.,
whenever $(G,T)$ is a reflection group whose set of
all reflections $T$ contains reflections $t$ of order $3$ or more, 
so that $\ell(t \cdot t)=\ell(t)=1$.

\begin{proposition}[Exchange property]
\label{prop:exchange}
Let $(G, T)$ be a generated group with $T=T^{-1}$
and $T$ closed under $G$-conjugation.
If $\ell(tg)<\ell(g)$ for some $t\in T$ and $g$ in $G$,
then there is a $T$-reduced word $g=t_1\cdots t_k$ 
such that $tg={}^tt_1\cdots {}^tt_{i-1} \cdot t_{i+1}\cdots t_k$.
\end{proposition}
\begin{proof}
If $\ell(tg)<\ell(g)$ for some $t\in T$ 
then Proposition~\ref{prop:lengthpm} implies $\ell(tg)=\ell(g)-1$.
Hence $t^{-1}\leq g$  and the 
subword property (Proposition~\ref{prop:subword}) 
implies that $t^{-1}$ is a subword of $(t_1,\ldots,t_k)$ for
some $T$-reduced expression $g=t_1\cdots t_k$.  If $t_i=t^{-1}$, then
\[
tg=tt_1\cdots t_{i-1}t^{-1}t_{i+1}\cdots t_k
  ={}^tt_1\cdots {}^tt_{i-1} \cdot t_{i+1}\cdots t_k.
\qedhere
\]
\end{proof}


\section{Counting ranks in the absolute order on
$\GL_n(\FF_q)$}
\label{ranks-in-whole-order}

When the field $\FF=\FF_q$ is finite, so that $\GL_n:=\GL_n(\FF_q)$ is
finite, it is easy to give an explicit formula and generating
function counting elements at rank $k$ in the absolute order on $\GL_n$, that is, 
those having fixed space codimension $k$.  Such a formula,
equivalent to \eqref{Fulman-formula} below,
was derived\footnote{Fulman credits its first proof to
unpublished work of Rudvalis and Shinoda \cite{RudvalisShinoda}.}
 in work of Fulman \cite[Thm.~6(1)]{Fulman} in a probabilistic context.

In the formula and elsewhere, we will use some standard $q$-analogues:
\begin{align*}
(x;q)_n&:=(1-x)(1-xq)(1-xq^2) \cdots (1-xq^{n-1}),\\
[n]_q & :=1+q+q^2+\cdots+q^{n-1},\\
[n]!_q & :=[1]_q [2]_q \cdots [n]_q=\frac{(q;q)_n}{(1-q)^n},\\
\qbin{n}{k}{q} & :=\frac{[n]!_q}{[k]!_q [n-k]!_q}
                       =\frac{(q;q)_n}{(q;q)_k (q;q)_{n-k}}
                       =\#\{k\text{-dimensional }\FF_q\text{-subspaces of }V=\FF_q^n\}.
\end{align*}
We mention for future use the fact that
\begin{equation}
\label{GL-cardinality}
|\GL_n(\FF_q)| =(q^n-1)(q^n-q)(q^n-q^2)\cdots (q^n-q^{n-1}) 
=(-1)^n q^{\binom{n}{2}} (q;q)_n
\end{equation}
as well as the $q$-binomial theorem \cite[(1.87)]{EC1}:
\begin{equation}
\label{q-binomial-theorem}
(x;q)_n = \sum_{k=0}^n (-1)^k q^{\binom{k}{2}} \qbin{n}{k}{q} x^k.
\end{equation}

\begin{proposition}
\label{abs-order-rank-sizes}
The number of $g$ in $\GL_n:=\GL_n(\F_q)$ having rank $k$ in absolute order is
\begin{align}
\label{formula for r(n, k)}
r_q(n, k) 
 &:= (-1)^k q^{{k \choose 2}} \qbin{n}{k}{q} 
      \sum_{j=0}^k  \qbin{k}{j}{q} q^{j(n-k)}   (q;q)_j \\
\label{Fulman-formula}
&= \frac{|\GL_n|}{|\GL_{n-k}|}
     \sum_{j=0}^k \frac{ (-1)^j q^{\binom{j}{2}-j(n-k)} }{|\GL_j|},
\end{align}
with generating function
\begin{equation}
\label{gf for r(n, k)}
1+\sum_{n \geq 1} 
\left( 
 \sum_{0 \leq k \leq n} r_q(n,k)
    x^{n - k}
\right)  \frac{y^n}{|\GL_n|}  = \frac{1}{1 - y}  \sum_{n \geq 0} \frac{(x; q^{-1})_n}{(q; q)_n}  y^n.
\end{equation}
\end{proposition}

\begin{proof}
The equivalence of formulas \eqref{formula for r(n, k)} and \eqref{Fulman-formula}
is a straightforward exercise using \eqref{GL-cardinality}.  Thus we will derive \eqref{formula for r(n, k)}, and then check that it agrees with \eqref{gf for r(n, k)}.

By Proposition~\ref{length=codim}, 
we need to count elements in $\GL_n$
whose fixed subspace has codimension $k$. 
For a subspace $W$ of $V=\F_q^n$, let 
\begin{align*}
g(W)&:=|\{g \in G: V^g = W\}|,\\
f(W)&:=|\{g \in G: V^g \supseteq W\}|=\sum_{U\supseteq W}g(U),
\end{align*}
so that if $\codim(W)=k$ one has 
\begin{align}
\label{q-binomial-inflation}
r_q(n,k) &= \qbin{n}{k}{q} g(W),\\
\notag
f(W) &= q^{k(n-k)} |\GL_k|= q^{k(n-k)} \cdot (-1)^k q^{\binom{k}{2}} (q;q)_k.
\end{align}
M\"obius inversion \cite[Ex.~3.10.2]{EC1} in the 
lattice of subspaces of $\Fq^n$ gives for $\codim(W)=k$, 
\[
g(W) = \sum_{U \supseteq W} \mu(W,U) f(U) \\
=\sum_{j=0}^k \qbin{k}{j}{q} (-1)^{k-j}
q^{{k-j\choose 2}} \cdot (-1)^j q^{j(n-j)+{j \choose 2}}(q;q)_j
\]
from which \eqref{formula for r(n, k)} 
follows via \eqref{q-binomial-inflation}.

To check \eqref{gf for r(n, k)}, use \eqref{q-binomial-theorem} to see that the coefficient of $y^n$ on its right is
\[
\sum_{m=0}^n \frac{(x;q^{-1})_m}{(q;q)_m}
= \sum_{m=0}^n \frac{1}{(q;q)_m} \sum_{i=0}^m (-1)^i q^{-\binom{i}{2}} \qbin{m}{i}{q^{-1}} x^i.
\]
Therefore the coefficient of $y^n x^{n-k}$
on the right of \eqref{gf for r(n, k)} equals
\[
(-1)^{n-k} q^{-\binom{n-k}{2}}
\sum_{m=n-k}^n \frac{ 1 }{(q;q)_m} \qbin{m}{n-k}{q^{-1}}.
\]
Reindexing $j:=n-m$ in the summation, and
using the fact that
$$
\qbin{a+b}{a}{q^{-1}}
= q^{-ab} \qbin{a+b}{a}{q},
$$
one then finds that the coefficient of  $y^n x^{n-k} / |\GL_n|$ on the right 
of \eqref{gf for r(n, k)}  equals
$$
|\GL_n| \cdot (-1)^{n-k} q^{-\binom{n-k}{2}}
\sum_{j=0}^k \frac{ q^{-(n-k)(k-j)} }{(q;q)_{n-j}}  \qbin{n-j}{n-k}{q} 
= (-1)^{k} q^{\binom{k}{2}} \qbin{n}{k}{q}
\sum_{j=0}^k \qbin{k}{j}{q} q^{j(n-k)} (q;q)_{j},
$$
which agrees with the formula \eqref{formula for r(n, k)} for $r_q(n,k)$.
\end{proof}

\begin{remark}
\label{r(n, k) vs. Stirling}
The formula~\eqref{formula for r(n, k)} for $r_q(n, k)$ is reminiscent of the inclusion-exclusion formula 
\[
\binom{n}{k} \sum_{j=0}^k (-1)^{j} \binom{k}{j} (k-j)! 
\]
counting permutations with $n - k$ fixed points.  On the other hand, it seems more natural to think of $r_q(n,k)$ as a $q$-analogue of $c(n,n-k)$,
the \emph{signless Stirling number of the first kind}, counting
the permutations in the symmetric group $\symm_n$ having $n-k$ cycles:  
when $\symm_n$ acts as a real reflection
group permuting coordinates in $V=\RR^n$, 
a permutation $\sigma$ with $n-k$ cycles has $\codim(V^\sigma)=k$.
In this sense, Equation~\eqref{gf for r(n, k)} gives a $q$-analogue of the formula
\[
1+\sum_{n \geq 1} \sum_{0 \leq k \leq n} c(n,n-k)
    x^{n - k} \frac{y^n}{n!}  = (1 - y)^{-x} = \sum_{k = 0}^\infty (-1)^k \binom{-x}{k} y^k,
\]
particularly when one observes that 
$\dfrac{(x; q^{-1})_k}{(q; q)_k} = \qbin{N}{k}{q}$ if $x = q^N$.
\end{remark}


\section{Counting chains below a Singer cycle in 
$\GL_n(\F_q)$}
\label{section:chains}

In the theory of finite 
irreducible real reflection groups, the interval $[e,c]$ in absolute order
below a \emph{Coxeter element} $c$ is sometimes
called the poset $NC(W)$ of \defn{$W$-noncrossing partitions}.
It is extremely well-behaved from several enumerative points of view,
including pleasant formulas for its cardinality, its
\emph{M\"{o}bius function}, and its \emph{zeta polynomial}.  In the classical
types $A,B/C,D$ one additionally has formulas for the following
more refined counts; see Edelman \cite[Thm.~3.2]{Edelman} for type $A$, 
Reiner \cite[Prop.~7]{Reiner} for types $B/C$, and Athanasiadis--Reiner \cite[Thm.~1.2(ii)]{AthanasiadisReiner} for type $D$.

\begin{definition}
\label{flag-f-definition}
Fix a reflection group $G$ having an absolute order,
and an element $c$ of $G$ with $\ell(c)=n$.  
The {\it flag $f$-vector} $(f_\alpha)$ of the interval $[e,c]$ has entries 
$f_\alpha:=f_\alpha[e,c]$ indexed by {\it compositions}
$\alpha=(\alpha_1,\ldots,\alpha_m)$ of $n=\sum_i \alpha_i$ with $\alpha_i > 0$. 
The entry $f_\alpha[e,c]$ is the number of chains 
\[
e=c_0 < c_1 < c_2 < \cdots < c_{m-1} < c_m=c
\]
in which $c_i$ has rank $\alpha_1+\alpha_2+\cdots+\alpha_i$ for each $i$.

Since $c_{i-1} < c_{i}$ if and only if $g_i:=c_{i-1}^{-1} c_i$ has 
$\ell(g_i)=\ell(c_i)-\ell(c_{i-1})$, one can rephrase the definition as
\[
f_\alpha[e,c]=\Big|\{(g_1,\ldots,g_m) \in G^m \colon c=g_1 \cdots g_m, \text{ and }\ell(g_i)=\alpha_i \text{ for each }i\}\Big|.
\] 
\end{definition}

As mentioned in the introduction, when viewing $\GL_n(\FF_q)$ as a finite
reflection group, the role analogous to that of a Coxeter element is
played by a \defn{Singer cycle} $c$,
which is the image of a multiplicative generator 
for $\FF_{q^n}^\times$ after one embeds
$\FF_{q^n}^\times$ into $\GL_n(\FF_q) \cong \GL_{\FF_q}(\FF_{q^n})$; see \cite[\S 9]{CSP}, \cite[Thm.~19]{RStantonWebb}, \cite{LRS}. 
Our goal in this section is to prove 
an unexpectedly simple formula for
the flag $f$-vector $f_\alpha[e,c]$ when $c$ is a Singer cycle; see Theorem~\ref{thm:flag-f-vector} below.  The special case where $\alpha=(1,1,\ldots,1)$ 
appeared in Lewis--Reiner--Stanton~\cite{LRS}, where it was
shown that there are exactly $(q^n-1)^{n-1}$ maximal chains in $[e,c]$ (equivalently, minimal factorizations of a Singer cycle into reflections).

In fact, the theorem also confirms a special case\footnote{Theorem~\ref{thm:flag-f-vector} confirms the special case  \cite[Conj.~6.3 at $\ell=n$]{LRS}.
In forthcoming work \cite{LewisMorales}, the second author and Morales use the same techniques to confirm \cite[Conj.~6.3]{LRS} in full generality.} 
of \cite[Conj.~6.3]{LRS}: 
it applies not only to a Singer cycle $c$, but to any 
element $c$ in $\GL_n(\FF_q)$ which is \defn{regular elliptic}, meaning that $c$ stabilizes no proper subspaces
in $\FF_q^n$.  (Equivalently, regular elliptic elements are those that act on $V=\FF_q^n$ with characteristic polynomial which is irreducible in $\FF_q[x]$; see \cite[Prop.~4.4]{LRS} for other equivalent definitions.)

 To state the theorem, define the quantity
$$
\varepsilon(\alpha) := \sum_{i=1}^{m} (\alpha_i-1)(n-\alpha_i). 
$$

\begin{theorem}
\label{thm:flag-f-vector}
For any regular elliptic element  $c$ in $\GL_n(\FF_q)$ and any composition $\alpha = (\alpha_1, \ldots, \alpha_m)$ of $n$, one has
\begin{equation}
\label{q-flag-count}
f_\alpha[e,c] = q^{\varepsilon(\alpha)} \cdot (q^n - 1)^{m - 1}.
\end{equation}
In particular, the number of elements of $[e,c]$ of rank $k$ for $1 \leq k \leq n-1$ is 
\begin{equation}
\label{interval-rank-sizes}
f_{(k,n-k)}[e,c] = q^{2k(n-k)-n} \cdot (q^n - 1).
\end{equation}
\end{theorem}

We remark that
Theorem~\ref{thm:flag-f-vector} appears very reminiscent of a special case of Goulden and Jackson's \emph{cactus formula}, counting the \defn{genus zero} 
factorizations $\sigma=\sigma_1 \cdots \sigma_m$ of an $n$-cycle~$\sigma$;  these are
the factorizations which are additive  $\sum_{i=1}^m \ell(\sigma_i) = \ell(\sigma)$
for the absolute length function given by 
$
\ell(\tau)=\sum_j({\lambda_j-1})
$
if $\tau$ has cycle sizes $(\lambda_1,\lambda_2,\ldots)$.  (This is the same length function discussed in Remark~\ref{r(n, k) vs. Stirling}.)
To state it, we need the following notation: given a partition $\lambda=1^{m_1} 2^{m_2} 3^{m_i} \cdots$
having $m_i$ parts of size $i$ and $m:=\sum_i m_i$ parts total, define
\[
N(\lambda)=\frac{1}{m}\binom{m}{m_1,m_2,\ldots}.
\]
If $\lambda=(\lambda_1,1^{n-\lambda_1})$ corresponds to a permutation
with only one nontrivial cycle then $N(\lambda)=1$.

\begin{theorem}[Cactus formula {\cite[Thm.~3.2]{GouldenJackson}}]
\label{cactus-theorem}
For an $n$-cycle $\sigma$ in the symmetric group $\symm_n$, the number of factorizations
$\sigma=\sigma_1 \cdots \sigma_m$ that
\begin{compactitem}
\item are additive, i.e., $\sum_i \ell(\sigma_i)=n-1(=\ell(\sigma))$, and
\item have $\sigma_i$ with cycle sizes $(\lambda^{(i)}_1,\lambda^{(i)}_2,\ldots)= \lambda^{(i)}$
\end{compactitem}
is given by 
\[
n^{m-1} \prod_{i=1}^m  {N(\lambda^{(i)})}.
\]
In particular, in the special case where each $\sigma_i$ has only one nontrivial cycle,
 the number of factorizations is given by
\begin{equation}
\label{cactus-special-case}
n^{m-1}.
\end{equation}
\end{theorem}

We currently lack a combinatorial proof of Theorem~\ref{thm:flag-f-vector};
see Question~\ref{Biane-type-proof-question}.
Instead, prompted by the similarity between \eqref{q-flag-count} and \eqref{cactus-special-case}, we prove the former by following a $q$-analogue 
of a proof of the latter due to Zagier; see \cite[\S A.2.4]{LandoZvonkin}.  
We sketch here the steps in Zagier's proof and give the $q$-analogous steps in the subsections below.

The first step is the same for both proofs, namely a 
representation-theoretic approach to 
counting factorizations that goes back to Frobenius;  see, e.g.,  \cite[\S A.1.3]{LandoZvonkin} for a proof.

\begin{definition}
Given a finite group $G$, let 
$\Irr(G)$ be the set of its irreducible ordinary (finite-dimensional,
complex) representations $U$.  For each $U$ in $\Irr(G)$, define
its \defn{character} $\chi_U(-)$, \defn{degree} $\chi_U(e)$, and
\defn{normalized character} $\normchi_U(-)$ by
\begin{align*}
\chi_{U}(g)&:=\Tr(g: U \rightarrow U), \\
\chi_U(e)&=\dim_\CC U,\\
\normchi_{U}(g)&:=\frac{\chi_U(g)}{\chi_U(e)}.
\end{align*}
Both functions $\chi_{U}(-), \normchi_{U}(-)$ on $G$
extend $\CC$-linearly to functions on the \defn{group algebra} $\CC[G]$. 
\end{definition}

In the sequel, we will frequently conflate a representation $U$ with its character $\chi_U$ without comment.

\begin{proposition}[Frobenius \cite{Frobenius}]
\label{Frobenius-prop}
Let $G$ be a finite group and let $A_1,\ldots,A_m \subset G$ be
unions of conjugacy classes in $G$. Let $z_i:=\sum_{g_i \in A_i} g_i$ in $\CC[G]$. 
Then for each $g$ in $G$,
\begin{equation}
\label{Chapuy-Stump-varying-class-answer}
|\{(g_1,\ldots,g_m) \in A_1 \times \cdots \times A_m \colon 
  g=g_1 \cdots g_m\}|
=\frac{1}{|G|} \sum_{\chi \in \Irr(G)} 
  \chi(e) \chi(g^{-1}) \prod_{i=1}^m \normchi(z_i).
\end{equation}
\end{proposition}

Zagier's proof of Theorem~\ref{cactus-theorem} applies Proposition~\ref{Frobenius-prop} by following these four steps.

\subsection*{Step 1}
One observes that, when applying  \eqref{Chapuy-Stump-varying-class-answer} to count factorizations of an $n$-cycle in $G=\symm_n$, the summation is much \emph{sparser} than it looks initially.  Irreducible $\symm_n$-characters $\chi^\lambda$ are indexed by partitions $\lambda$ of $n$, but the only $\chi^\lambda$ which do not vanish on  an $n$-cycle $\sigma$ are
the \defn{hook shapes}, i.e., those of the form $\lambda=(n-d,1^d)$ for $d=0,1,\dots,n-1$.  These satisfy
\[
\chi^{(n-d,1^d)}(\sigma) =(-1)^d   \qquad \textrm{ and } \qquad
\chi^{(n-d,1^d)}(e) =\binom{n-1}{d}.
\]
Hence Proposition~\ref{Frobenius-prop} shows that the number of additive factorizations $\sigma=\sigma_1 \cdots \sigma_m$ in which each $\sigma_i$ has cycle type $\lambda^{(i)}$ is
\begin{equation}
\label{Zagier-sparse-expression}
\frac{1}{n!} \sum_{d=0}^{n-1} (-1)^d \binom{n-1}{d}  P(d),
\qquad\text{ where }\quad
P(d):=\prod_{i=1}^m \normchi^{(n-d,1^d)}(z_i)
\end{equation}
and each $z_i$ is the sum in $\CC[\symm_n]$ of all permutations of cycle type $\lambda^{(i)}$.

\subsection*{Step 2}
One shows that each normalized character value $\normchi^{(n-d,1^d)}(z_i)$ appearing as a factor in \eqref{Zagier-sparse-expression} is the specialization at $x=d$ of a polynomial $P_{\lambda^{(i)}}(x)$ in $\QQ[x]$.  This polynomial has degree $\sum_j (\lambda^{(i)}_j-1)$ and a predictable, explicit leading coefficient.  
Thus the product $P(d)$ is also the specialization 
of a polynomial $P(x)$ in $\QQ[x]$, having degree $n - 1$ and a predictable, explicit leading coefficient.

\subsection*{Step 3}
Note that the $N$th iterate $\Delta^N:=\Delta \circ \cdots \circ \Delta$ of the forward difference operator
\begin{equation}
\label{forward-difference}
\Delta(f)(x):=f(x+1)-f(x)
\end{equation}
satisfies
\begin{equation}
\label{forward-difference-iterate}
(\Delta^N f)(x)=\sum_{d=0}^N (-1)^d \binom{N}{d} f(x+d).
\end{equation}
Hence the sum \eqref{Zagier-sparse-expression} 
is the $(n-1)$st forward difference of $P(x)$ evaluated at $x = 0$, that is,
$
(\Delta^{n-1}P)(0).
$

\subsection*{Step 4}
For each integer $m \geq 0$ one has
\[
\Delta(x^m)=(x+1)^m-x^m=mx^{m-1}+O(x^{m-2}),
\]
and so the operator $\Delta$ lowers degree by $1$ and scales by $m$
the leading coefficient of a degree-$m$ polynomial.  Hence the polynomial $P(x)$ from Step 2 has
$\Delta^{n-1}P=(\Delta^{n-1}P)(0)$ equal to
a constant, namely $(n-1)!$ times the leading coefficient of $P(x)$. 
Thus our answer \eqref{Zagier-sparse-expression}, which is equal to
$\frac{1}{n!} (\Delta^{n-1}P)(0)$ by Step 3, is $\frac{(n-1)!}{n!}=\frac{1}{n}$ times
the leading coefficient of $P(x)$ computed in Step 2.

\medskip

In the next four subsections, we describe what we view as $q$-analogues of
Steps 1, 2, 3, 4, in order to prove Theorem~\ref{thm:flag-f-vector}.
As a preliminary step, take $\GL_n:=\GL_n(\Fq)$, 
acting on $V = \Fq^n$, and define
for $k=0,1,\ldots,n$ the element $z_k$ 
in $\CC[\GL_n]$ to be the sum of all elements $g$ for
which $\codim(V^g)=k$.  Then 
Definition~\ref{flag-f-definition} and Proposition~\ref{Frobenius-prop} show that
\begin{equation}
\label{chain counting equation}
 f_\alpha[e,c] = \frac{1}{|\GL_n|} \sum_{\chi \in \Irr(\GL_n)} \chi(e) \chi(c^{-1}) 
\prod_{i=1}^m \normchi(z_{\alpha_i}).
\end{equation}

\subsection{A $q$-analogue of Step 1.}

Just as in Step 1 above, one observes that for a regular elliptic element $c$ in $\GL_n$, the summation  \eqref{chain counting equation} is much {\it sparser} than it looks initially, as many $\GL_n$-irreducibles have
$\chi(c^{-1})=0$.

To explain this, we begin with a brief outline of some of the theory of complex characters of $\GL_n(\FF_q)$.  The theory was first developed by J.A. Green \cite{Green}, building on R. Steinberg's work \cite{Steinberg} constructing the unipotent characters $\chi^{\one,\lambda}$.   It has been reworked several times, e.g., by  Macdonald \cite[Chs.~III, IV]{Macdonald} and Zelevinsky \cite[\S 11]{Zelevinsky}.

\begin{definition}
A key notion is the {\it parabolic} or {\it Harish-Chandra induction}  $\chi_1 \hcprod \chi_2$ of
two characters $\chi_1, \chi_2$ for $\GL_{n_1}, \GL_{n_2}$ to give a character of
$\GL_n$ where $n=n_1+n_2$.  To define it, introduce the parabolic subgroup 
\begin{equation}
\label{parabolic-definition}
P_{n_1,n_2}:=\left\{
\left[
\begin{matrix} A_1 & B\\ 0 & A_2 \end{matrix}
\right] \text{ in }\GL_n
\right\}
\end{equation}
so that $A_i$ lies in $\GL_{n_i}$ for $i=1,2$, and $B$ is arbitrary in $\FF_q^{n_1 \times n_2}$.  Then
\begin{equation}                                                                                                   
\label{parabolic-induction-formula}                                                                                
(\chi_1 \hcprod \chi_2)(g):=                                              
\frac{1}{|P_{n_1,n_2}| }
\sum_{
 \substack{h \in G\colon\\ hgh^{-1} \in P_{n_1,n_2}}}
\chi_1(A_1) \chi_2(A_2),
\end{equation}
where the element $hgh^{-1}$ of $P_{n_1,n_2}$ has
 diagonal blocks labeled $A_1,A_2$ as above.
Said differently, 
$
\chi_1 \hcprod \chi_2
  :=\left( \chi_1 \otimes \chi_2 \right) \Uparrow^{P_{n_1,n_2}}_{\GL_{n_1} \times \GL_{n_2}}
      \uparrow_{P_{n_1,n_2}}^{\GL_n}
$
where
\begin{itemize}
\item
$(-)\Uparrow_{\GL_{n_1} \times \GL_{n_2}}^{P_{n_1,n_2}}$ is {\it inflation} of
representations of $\GL_{n_1} \times \GL_{n_2}$ into those of $P_{n_1,n_2}$,
by precomposing with the surjection 
$P \twoheadrightarrow \GL_{n_1} \times \GL_{n_2}$, and
\item
$(-)\uparrow_{P_{n_1,n_2}}^{\GL_n}$ is {\it induction} of representations.
\end{itemize} 
\end{definition}

The binary operation $(\chi_1,\chi_2) \longmapsto \chi_1 \hcprod \chi_2$ turns out \cite[Ch.~III]{Zelevinsky} to define an associative, commutative (!), graded $\CC$-algebra structure on $\bigoplus_{n \geq 0} \Class(\GL_n)$, 
where $\Class(\GL_n)$ denotes the $\CC$-vector space of 
class functions on $\GL_n$, with $\Class(\GL_0):=\CC$.  

\begin{definition}
 An irreducible $U$ in $\Irr(\GL_n)$ is called \defn{cuspidal}, with \defn{weight} $\wt(U)=n$, if $U$ is not a constituent of any proper induction $\chi_1 \hcprod \chi_2$ for characters $\chi_i$ of $\GL_{n_i}$ with $n=n_1+n_2$ and $n_1, n_2 \geq 1$. 

Denote by $\Cusp_n$ the set of weight-$n$ cuspidal characters, and $\Cusp:=\sqcup_{n \geq 0} \Cusp_n$.  
\end{definition}

\begin{definition}
An irreducible $\GL_n$-character is called \defn{primary} to the cuspidal $U$ if $\chi$ \emph{does} occur as an irreducible constituent of some product $U^{\hcprod{\frac{n}{s}}}=U \hcprod U \hcprod \cdots \hcprod U$, where $\wt(U)=s$.
\end{definition}

\noindent
It turns out that one can parametrize the irreducible $\GL_n$-characters 
primary to the cuspidal $U$ as $\{ \chi^{U,\lambda}: |\lambda|=\frac{n}{s} \}$,
parallel to the parametrization of the irreducible $\symm_n$-characters
as $\{ \chi^\lambda: |\lambda| = n \}$.  In fact, two primary 
irreducibles $\chi^{U,\mu}, \chi^{U,\nu}$ for $\GL_{n_1},\GL_{n_2}$
primary to the same cuspidal $U$ have product controlled by
the usual \emph{Littlewood--Richardson coefficients}:
\[
\chi^{U,\mu} \hcprod \chi^{U,\nu} 
  = \sum_\lambda c_{\mu,\nu}^\lambda \chi^{U,\lambda}
\quad \text{ where }\quad
(\chi^{\mu} \otimes \chi^{\nu} )
  \uparrow_{ \symm_{|\mu|} \times \symm_{|\nu|} }^{\symm_{|\mu|+|\nu|}}
  = \sum_\lambda c_{\mu,\nu}^\lambda \chi^{\lambda}.
\]
Furthermore, the set of \emph{all} irreducibles
$\Irr(\GL_n)$ can be indexed as $\{ \chi^{\llambda} \}$ in
which $\llambda$ runs through the
functions $\llambda:U \longmapsto \lambda(U)$
from $\Cusp$ to all integer partitions,
subject to the restriction
$\sum_U \wt(U) \cdot |\lambda(U)|=n$.
In this parametrization,
\[
\chi^{\llambda} =  \chi^{U_1,\lambda(U_1)}\hcprod \cdots \hcprod\chi^{U_m,\lambda(U_m)}
\]
if $\{U_1,\ldots,U_m\}$ are the cuspidals having $\lambda(U_i) \neq \varnothing$.

We next recall from \cite{LRS} the sparsity statement analogous to that of Step 1, showing that most irreducible $\GL_n$-characters $\chi$ vanish on a regular elliptic element.  We also include the character values and a degree formula for certain irreducibles that arise in our computation.

\begin{proposition}[{\cite[Prop.~4.7]{LRS}}]
\label{Singer-cycle-character-values}
Let $c$ in $\GL_n$ be regular elliptic, e.g., a Singer cycle. 
\begin{compactenum}[(i)]
\item The irreducible character $\chi^{\llambda}$ has vanishing 
value $\chi^{\llambda}(c) = 0$ unless $\chi$ is a primary irreducible
$\chi^{U,\lambda}$ for some cuspidal $U$ with $\wt(U)=s$ dividing $n$,
and  $\lambda = \hook{d}{\frac{n}{s}}$ is a hook-shaped partition of $\frac{n}{s}$.  
\item If $U = \one=\one_{\GL_1}$ is the trivial cuspidal with $s=\wt(U)=1$, then 
\[
\chi^{\one, \hook{d}{n}}(c) = (-1)^d \qquad \textrm{ and } \qquad
\chi^{\one, \hook{d}{n}}(e) = q^{\binom{d+1}{2}} \qbin{n-1}{d}{q}.
\]
\end{compactenum}
\end{proposition}

\subsection{A $q$-analogue of Step 2.}

Of course, to use \eqref{chain counting equation} we also 
need some character values on the elements $z_k$.  These are provided by
the following remarkable result, which was suggested by computations in GAP \cite{GAP}.  
Its proof is deferred to Appendix~\ref{technical-proof}.

\begin{proposition}
\label{character values prop}
One has these normalized character values on $z_k$ for certain $\normchi^{U,\lambda}$.
\begin{compactenum}[(i)]
\item
For any primary irreducible $\GL_n$-character $\chi^{U,\lambda}$ with
the cuspidal $U \neq \one$ nontrivial,
\[
\normchi^{U, \lambda}(z_k) = (-1)^k q^{\binom{k}{2}} \qbin{n}{k}{q}.
\]
\item
For $U = \one$ and $\lambda = \hook{d}{n}$ a hook, we have
\[
\normchi^{\one, \hook{d}{n}}(z_k) =\P_k(q^{-d})
\]
where $\P_k(x)$ is the following polynomial in $x$ of degree $k$:
\begin{equation}
\label{definition-of-P}
\P_k(x) := (-1)^k q^{\binom{k}{2}} \left(
\qbin{n}{k}{q} + 
\frac{1 - q^n}{[n - k]!_q}  
\sum_{j = 1}^k \frac{[n - j]!_q}{[k - j]!_q} q^{j(n - k)} x \cdot (x q^{n - j + 1}; q)_{j - 1}
\right).
\end{equation}
\end{compactenum}
\end{proposition}

\subsection{A $q$-analogue of Step 3.}

We are now well-equipped to analyze the summation in \eqref{chain counting equation} by breaking it into two pieces:
\begin{equation}
\label{A+B-expression}
f_\alpha[e,c] 
 = \frac{1}{|\GL_n|} \sum_{\substack{\chi \in \Irr(\GL_n)\colon\\ \chi(c^{-1}) \neq 0}} 
\chi(e) \chi(c^{-1}) 
\prod_{i=1}^m \normchi(z_{\alpha_i}) 
  = \frac{1}{|\GL_n|} (A + B) 
\end{equation}
where  
$A$ is the sum over primary irreducibles $\chi^{U,\lambda}$ with $U \neq \one =\one_{\GL_1}$ and
 $B$ is the sum over primary irreducibles of the form
$\chi^{\one,\lambda}$.  By Proposition~\ref{character values prop}(i), one has
\[
A=\prod_{i=1}^{m} (-1)^{\alpha_i} q^{\binom{\alpha_i}{2}} \qbin{n}{\alpha_i}{q}
   \;  
\sum_{\substack{\chi^{U,\lambda} \in \Irr(\GL_n)\colon\\ U \neq \one}} \chi^{U,\lambda}(e) \chi^{U,\lambda}(c^{-1}).
\]
However, Proposition~\ref{Singer-cycle-character-values}(i) lets one rewrite 
this last summation as
\[
 \sum_{\substack{\chi^{U,\lambda} \in \Irr(\GL_n)\colon\\ U \neq \one}} \chi^{U,\lambda}(e) \chi^{U,\lambda}(c^{-1})
 =
\sum_{\substack{\chi \in \Irr(\GL_n)\colon\\ \chi(c^{-1}) \neq 0}} 
\chi(e) \chi(c^{-1})
-
\sum_{d=0}^{n-1} \chi^{\one,(n-d,1^d)}(e) \chi^{\one,(n-d,1^d)}(c^{-1}).
\]
The first sum on the right side is the character of the \emph{regular representation} for $\GL_n$ evaluated at $c$, and hence is equal to $0$.  By  Proposition~\ref{Singer-cycle-character-values}(ii) and
 the $q$-binomial theorem \eqref{q-binomial-theorem}, the second sum on the right side is
\[
\sum_{d=0}^{n-1} (-1)^d q^{\binom{d+1}{2}} \qbin{n-1}{d}{q} = (q;q)_{n-1}.
\]
Thus one concludes that
\begin{equation}
\label{final-A-expression}
A= -(q;q)_{n-1} 
     \prod_{i=1}^{m} (-1)^{\alpha_i} q^{\binom{\alpha_i}{2}} \qbin{n}{\alpha_i}{q}.
\end{equation} 

Next we analyze the sum $B$ in \eqref{A+B-expression}.  For a composition $\alpha$, define
$
\P_\alpha(x) = \prod_i \P_{\alpha_i}(x).
$
By Propositions~\ref{Singer-cycle-character-values} and~\ref{character values prop} and the definition of $B$, we may rewrite
\begin{equation}
\label{first-expression-for-B}
B=
\sum_{d = 0}^{n - 1} (-1)^d q^{\binom{d + 1}{2}} \qbin{n - 1}{d}{q} \P_{\alpha}(q^{-d}).
\end{equation}
We identify $B$ in terms of the $(n-1)$st iterate of a 
$q$-difference operator $\Delta_q$.  This operator is the $q$-analogue 
of \eqref{forward-difference} defined by 
\[
\Delta_q(f)(x) = \frac{f(qx)-f(x)}{qx - x}
  =\frac{f(qx)-f(x)}{(q-1) x}.
\] 

One can check via the $q$-Pascal recurrence
\[
\qbin{N}{d}{q} = \qbin{N-1}{d}{q}+q^{N-d}\qbin{N-1}{d-1}{q}
\]
and induction that for $N \geq 0$, the $N$th iterate
$\Delta^N_q=\Delta_q \circ \cdots \circ \Delta_q$ has the following
expression:
\begin{equation}
\label{iterate-of-q-difference}
 \Delta_q^{N}(f)(x) = q^{-\binom{N}{2}} (q - 1)^{-N} 
    \sum_{d= 0}^{N}(-1)^{d} q^{\binom{d}{2}} \qbin{N}{d}{q} \frac{ f(q^{N - d}x) }{x^N}.
\end{equation}
(This is $q$-analogous to \eqref{forward-difference-iterate}.)  Taking $N = n - 1$ in \eqref{iterate-of-q-difference} and applying the operator to $\P_{\alpha}(x)/x$ gives
\begin{align*}
\left[
 \Delta_q^{n - 1}\left( \frac{ \P_{\alpha}(x)}{x} \right) 
\right]_{x = q^{1 -n}}
 &= 
 q^{-\binom{n-1}{2}} (q - 1)^{1-n} 
    \sum_{d= 0}^{n-1}(-1)^{d} q^{\binom{d}{2}} \qbin{n-1}{d}{q}
   \left[
    \frac{1}{x^{n-1}}
     \frac{ \P_\alpha(q^{n-1 - d}x)}{(q^{n-1-d}x) }
   \right]_{x = q^{1 -n}} \\
&= 
 q^{-\binom{n-1}{2}+(n-1)^2} (q - 1)^{1-n} 
    \sum_{d= 0}^{n-1}(-1)^{d} q^{\binom{d+1}{2}} \qbin{n-1}{d}{q}
      \P_\alpha(q^{- d}).
\end{align*}
Combining with \eqref{first-expression-for-B} gives
\begin{equation}
\label{second-expression-for-B}
B
=(q - 1)^{n - 1} q^{-\binom{n}{2}} 
  \left[ 
   \Delta_q^{n - 1}\left( \frac{\P_{\alpha}(x)}{x} \right)
  \right]_{x = q^{1 -n}}.
\end{equation}

\subsection{A $q$-analogue of Step 4.}

We process the expression \eqref{second-expression-for-B} for $B$ further.  It is easily verified by induction on $N \geq 0$ that
 for any $m$, 
\[
\Delta_q^N(x^m) 
= \frac{(q^m-1)(q^{m-1} -1) \cdots (q^{m-N+1}-1)}{(q-1)^N} \cdot x^{m - N}
= \frac{(q^{m - N+1} ; q)_N}{(1-q)^N} \cdot x^{m - N}.
\] 
In particular, for integer $m$ one has
\begin{equation}
\label{relevant-q-diffs}
\Delta_q^N(x^m)=
\begin{cases}
0 & \text{  if } N > m \geq 0,\\
[m]!_q  & \text{  if } N = m \geq 0,\\
(-1)^N q^{\binom{N+1}{2}} [N]!_q \cdot x^{-N-1} 
  & \text{  if } m =-1.
\end{cases}
\end{equation}

\begin{proposition}
\label{boundary-terms-of-P}
For any composition
$\alpha=(\alpha_1,\ldots,\alpha_m)$ of $n$, 
the function $\P_\alpha(x)=\prod_{i} \P_{\alpha_i}(x)$ 
\begin{itemize}
\item is a polynomial in $x$ of degree $n$,
\item has leading coefficient equal to 
    $\displaystyle  q^{\varepsilon(\alpha)+n(n-1)} \cdot (q^n-1)^m$, and
\item has constant coefficient equal to $-A/ (q;q)_{n-1}$.
\end{itemize}
\end{proposition}
\begin{proof}
Note from the definition \eqref{definition-of-P} of $\P_k(x)$ that it is a polynomial in $x$ of degree $k$, with constant coefficient $(-1)^k q^{\binom{k}{2}} \qbin{n}{k}{q}$.  Hence $\P_\alpha(x)$
is polynomial in $x$ of degree $\sum_i \alpha_i=n$ with constant coefficient
\[
\prod_{i=1}^m 
(-1)^{\alpha_i} q^{\binom{\alpha_i}{2}} \qbin{n}{\alpha_i}{q} 
  = \frac{-A}{(q;q)_{n-1}},
\]
where the last equality uses \eqref{final-A-expression}.
One sees that in \eqref{definition-of-P}, 
the $x^k$ coefficient in $\P_k(x)$ is entirely accounted for 
by the $j=k$ summand,  and is equal to
\[
(-1)^k q^{ \binom{k}{2} + k(n-k) + \sum_{j=n-k+1}^{n-1} j} \cdot(q^n-1)
= q^{k(n-k)+n(k-1)} \cdot (q^n-1).
\]
Therefore the product $\P_\alpha(x) = \prod_{i} \P_{\alpha_i}(x)$
has leading coefficient 
\[
q^{\sum_i \alpha_i(n-\alpha_i)+n(\alpha_i-1) } \cdot (q^n-1)^m
=q^{\varepsilon(\alpha)+n(n-1)} \cdot (q^n-1)^m 
. \qedhere
\]
\end{proof}

As $\P_{\alpha}(x)$ has degree $n$ in $x$, the quotient
$\frac{\P_\alpha(x)}{x}$ is a Laurent polynomial with top degree $n-1$
and bottom degree $-1$.  Therefore, combining 
Proposition~\ref{boundary-terms-of-P} 
with \eqref{relevant-q-diffs} gives
\begin{align*}
\Delta_q^{n-1}\left( \frac{\P_\alpha(x)}{x} \right)
&=
(-1)^{n-1} q^{-\binom{n}{2}} [n-1]!_q \cdot x^{-n} \cdot \frac{-A}{(q;q)_{n-1}}
 \quad + \quad [n-1]!_q q^{\varepsilon(\alpha)+n(n-1)} \cdot (q^n-1)^m\\
&=[n-1]!_q 
  \left( 
     (-1)^{n-1} q^{-\binom{n}{2}} \frac{-A}{x^n \cdot (q;q)_{n-1}}
    \quad + \quad q^{\varepsilon(\alpha)+n(n-1)} \cdot (q^n-1)^m
  \right).
\end{align*}
Plugging this into \eqref{second-expression-for-B} and using
$(q-1)^{n-1}[n-1]!_q = (-1)^{n-1}(q;q)_{n-1}$ gives
\begin{align*}
B 
&=  (-1)^{n-1} q^{-\binom{n}{2}} (q;q)_{n-1} 
  \left( 
     (-1)^{n-1} q^{-\binom{n}{2}} \frac{-A}{q^{-n(n-1)} (q;q)_{n-1}}
     \quad + \quad q^{\varepsilon(\alpha)+n(n-1)} \cdot (q^n-1)^m
  \right)\\
&=-A  \quad + \quad  (-1)^{n-1} (q;q)_{n-1}  q^{\varepsilon(\alpha)+\binom{n}{2}}\cdot (q^n-1)^m.
\end{align*}
Using \eqref{GL-cardinality},
one can finally compute from \eqref{A+B-expression} that
\[
f_\alpha[e,c] 
  =\frac{1}{|\GL_n|} (A + B) 
 =\frac{(-1)^{n-1} (q;q)_{n-1} q^{\varepsilon(\alpha)+\binom{n}{2}}} 
            {(-1)^n (q;q)_{n} q^{\binom{n}{2}} } \cdot (q^n-1)^m
 =  q^{\varepsilon(\alpha)} \cdot (q^n-1)^{m-1}.
\]
This concludes the proof of Theorem~\ref{thm:flag-f-vector}. \hfill$\qed$

\vskip.1in
The preceding proof is 
computational and unenlightening.  This prompts the following question.

\begin{question}
\label{Biane-type-proof-question}
Biane \cite{Biane} has given a short, inductive proof of 
\eqref{cactus-special-case} not relying on any auxiliary objects (trees, maps, etc.).  Is there an analogous proof of Theorem~\ref{thm:flag-f-vector}?
\end{question}

\begin{question}
Is there a reasonable $q$-analogue of the cactus formula (Theorem~\ref{cactus-theorem}) in full generality, not just in the special case \eqref{cactus-special-case}?
\end{question}
\noindent
We currently have no conjectural candidate for such a $q$-analogue.

\section{Reformulating the flag $f$-vector}
\label{reformulation-section}

The goal of this section is to prove Proposition~\ref{chain-subspace proposition}, 
a linear algebraic reformulation of
$f_{\alpha}[e,c]$ when $V^c=0$.  We hope that this reformulation may be more amenable to 
combinatorial counting methods.  In particular, we show below that
it helps recover somewhat more directly the rank sizes
for $[e,c]$ given in \eqref{interval-rank-sizes}

\begin{definition}
Fix a field $\FF$, and let $V$ be an $n$-dimensional $\FF$-vector space.

Given a sequence
$g_\bullet:=(g_0, g_1, \ldots,g_{m-1},g_m)$ with $g_i$ in $\GL(V)$,
define a sequence of subspaces 
$
\varphi(g_\bullet):=(V_1,\ldots,V_m)
$
via 
\[
V_i:=V^{g_{i-1}} \cap V^{g_i^{-1}g_{m}}.
\]
  
Fix $c$ in $\GL(V)$.
Given an ordered vector space decomposition 
$V_\bullet=(V_i)_{i=1}^m$ of $V$, 
so that
\[
V=\underbrace{V_1 \oplus V_2 \oplus \cdots \oplus V_i}_{=:V_{\leq i}}
\oplus
\underbrace{V_{i+1} \oplus V_{i+2} \oplus \cdots \oplus V_m}_{=:V_{> i}},
\]
define a sequence 
$
\psi(V_\bullet):=(g_0,g_1,\ldots,g_{m-1},g_m)
$
of $\FF$-linear maps $g_i: V \rightarrow V$ by 
\[
g_i(x + y)=c(x) +y \qquad \text{ for } x, y \text{ in } V_{\leq i}, V_{> i}, \text{ respectively}.
\]
\end{definition}

\begin{proposition}
\label{chain-subspace proposition}
Let $V=\FF^n$ for a field $\FF$, and let $c$ lie in $G := \GL(V)$ with $V^c=0$.
Then the maps $\varphi, \psi$ restrict to inverse bijections
between these two sets:
\begin{compactenum}[(a)]
\item
multichains $g_\bullet:=(e=g_0 \leq g_1 \leq \cdots \leq g_{m-1} \leq g_m=c)$
in absolute order on $G$, and
\item
decompositions $V_\bullet=(V_i)_{i=1}^m$
satisfying
$V=c(V_{\leq i}) \oplus V_{> i}$ for every $i=0,1,\ldots,m$.
\end{compactenum}
Moreover, they satisfy $\dim V_i = \ell(g_i)-\ell(g_{i-1})$.

In particular, when $\FF=\FF_q$ is finite, for any composition $\alpha = (\alpha_1, \ldots, \alpha_m)$ of $n$, the flag number
$f_\alpha[e,c]$ counts decompositions $V_\bullet$ as in (b) having $\dim V_i= \alpha_i$ for $i=1,2,\ldots,m$.
\end{proposition}

\begin{proof}
Given $g_\bullet$ as in (a), we wish to
show that $\phi(g_\bullet)=(V_1,\ldots,V_m)$ is as in (b).
First note that Proposition~\ref{prop:self-duality} and 
$e \le g_{i-1}\leq g_i\le c$ imply $g_i^{-1}c\le g_{i-1}^{-1} c$.
Thus, from Remark~\ref{alternate-characterization-remark} we have
\begin{equation}
\label{fixed subspaces}
\begin{array}{rcccl}
V&=&V^{g_i}&\oplus&V^{g_i^{-1}c}\\
 & &\cap  &      &\cup\\
V&=&V^{g_{i-1}}&\oplus& V^{g_{i-1}^{-1}c}.
\end{array}
\end{equation}

As a first goal, we show $V=\bigoplus_{i=1}^m V_i$ 
via induction on $m$, with the base case $m=1$ being trivial.
In the inductive step, remove $g_1$ from $g_\bullet$ to give  
$g'_\bullet=(e\leq g_2 \leq \cdots \leq g_{m-1} \leq c)$. 
Then $\varphi(g'_\bullet)=(U_2,V_3,V_4,\ldots,V_m)$ satisfies
$V=U_2 \oplus \left(\bigoplus_{i=3}^m V_i\right)$ by induction.  
Moreover, note 
\[
U_2 = V \cap V^{g_2^{-1}c}
 = (V^{g_1^{-1}c} \oplus V^{g_1}) \cap V^{g_2^{-1}c}
 = V^{g_1^{-1}c}\oplus (V^{g_1} \cap V^{g_2^{-1}c})
 = V_1 \oplus V_2,
\]
where the second-to-last equality uses
$V^{g_1^{-1}c} \subset V^{g_2^{-1}c}$ from \eqref{fixed subspaces}. 
Hence $V=\bigoplus_{i=1}^m V_i$.  

We also claim $V_{\leq i}=V^{g_i^{-1}c}$ and $V_{> i}=V^{g_i}$.
To see this, note that for each $j \leq i$ one has
\[
V_j = (V^{g_{j-1}} \cap V^{g_j^{-1}c}) \subset V^{g_j^{-1}c} \subset V^{g_i^{-1}c}
\]
by \eqref{fixed subspaces}, 
and hence $V_{\leq i} \subseteq V^{g_i^{-1}c}$;  a similar argument
shows that $V_{> i} \subseteq V^{g_i}$.  But then
\[
V=V_{\leq i} \oplus V_{> i}=V^{g_i^{-1}c} \oplus V^{g_i}
\]
forces the claimed equalities, as well as $\dim(V_i)=\ell(g_i)-\ell(g_{i-1})$.
Lastly, applying $g_i$ to the decomposition in \eqref{fixed subspaces}, one obtains
the final desired property for (b):
\[
V=g_iV
=g_i(V^{g_i^{-1}c} \oplus V^{g_i})
=g_iV^{g_i^{-1}c} \oplus g_iV^{g_i}
=cV^{g_i^{-1}c} \oplus V^{g_i}
=cV_{\leq i} \oplus V_{> i}.
\]

Conversely, given $V_\bullet$ as in (b), we must show that
$\psi(V_\bullet)=g_\bullet$ is as in (a).   The assumption that
$V=cV_{\leq i} \oplus V_{> i}$ shows that $g_iV = V$, and hence each $g_i$ is
invertible.  

We claim $V^c=0$ shows $V^{g_i}=V_{> i}$: 
expressing
$v=x+y$ uniquely with $x,y$ in $V_{\leq i}, V_{> i}$, one has
$v$ in $V^{g_i}$ if and only if $c(x)+y=x+y$ if and only if $c(x)=x$ if and only if $x=0$.
Similarly, 
$
V^{g_{i-1}^{-1}g_i}=V_{\leq i-1} \oplus V_{> i}
$
.  Hence 
\[
\ell(g_{i-1})+\ell(g_{i-1}^{-1}g_i)=\dim V_{\leq i-1}+\dim V_i =
\dim(V_{\leq i})=\ell(g_i).
\]
Thus $g_{i-1} < g_i$ and so $g_\bullet$ satisfies (a).

Finally, one can check $\phi$ and $\psi$ are inverses to each other.
\end{proof}

\begin{proof}[Alternate proof of Equation \eqref{interval-rank-sizes},
via Proposition~\ref{chain-subspace proposition}]

Choose $c$ in $\GL(V)$ regular elliptic.  By Proposition~\ref{prop:self-duality}, it is enough to show that for $1\leq k\leq n/2$, there are
\[
f_{(k, n - k)}[e,c] = q^{\varepsilon((k, n - k))} \cdot (q^n-1)
                  = q^{ 2k(n-k) - n } \cdot (q^n-1)
\]
elements $g$ in $[e,c]$ of rank $k$.  By Proposition~\ref{chain-subspace proposition}, 
these elements are in bijection with direct sum decompositions
\[
V = \FF_q^n= U\oplus W = cU\oplus W
\]
where $\dim U = k$.  Count such decompositions by first choosing $U$, and
then choosing $W$ complementary to both $U$ and $cU$.  The number of choices of $W$ depends only on $k=\dim U$ and $d:=\dim (U \cap cU)$,
and thus it helps to have the following very special case of a general formula due to Chen and Tseng \cite[p.~28]{ChenTseng}: for
a regular elliptic element $c$ in $\GL_n(\FF_q)$, there  are
\[
g(n,k,d) := \frac{ [n]_q} { [k]_q} \qbin{n-k-1}{k-d-1}{q} \qbin{k}{d}{q} q^{(k-d)(k-d-1)}
\]
subspaces $U$ of $\Fq^n$ for which
$\dim U=k$ and 
$\dim (U \cap cU) = d$, assuming $0 \leq d < k < n$.

Given two $k$-dimensional subspaces $U_1, U_2$ of $V$ with
$\dim (U_1 \cap U_2) = d$ (such as $U_1=U$ and $U_2=cU$ above), 
it is a straightforward exercise to check that when $0 \leq d \leq k \leq n/2$ there are
\begin{equation}
\label{counting co-complements}
f(n, k, d) := q^{ k(n-k) - {k-d+1\choose 2}  } (-1)^{k-d}(q;q)_{k-d}
\end{equation}
subspaces $W$ with $V=U_1 \oplus W=U_2 \oplus W$.
Thus 
\begin{align*}
f_{\alpha}[e,c]&=\sum_{d = 0}^{k-1} g(n, k, d) f(n,k,d) \\
&=\sum_{d=0}^{k-1}
  \frac{ [n]_q} { [k]_q} \qbin{n-k-1}{k-d-1}{q} \qbin{k}{d}{q} q^{(k-d)(k-d-1)}
 \cdot     q^{ k(n-k) - {k-d+1\choose 2}  } (-1)^{k-d}(q;q)_{k-d}\\
&=(q^n-1) q^{k(n-k)-1} 
  \sum_{d=0}^{k-1} \qbin{k-1}{d}{q} (q^{n-k-1}-1)(q^{n-k-1}-q)\cdots (q^{n-k-1}-q^{k-d-2}).
\end{align*}
Finally, we apply the special case
\[
q^{ab} = \sum_{d=0}^a \qbin{a}{d}{q}  (q^b-1)(q^b-q) \cdots (q^b-q^{a-d-1})
\]
of the $q$-Chu--Vandermonde identity \cite[II.6]{GasperRahman} with $(a,b)=(k-1, n-k-1)$ to conclude.
\end{proof}
\begin{remark}
Both the Chen--Tseng result and the needed case of the $q$-Chu--Vandermonde identity have elementary proofs: in the former case by a complicated recursive argument, and in the latter case by counting matrices in $\FF_q^{a \times b}$ by their row spaces (see, e.g., \cite{Landsberg}).
\end{remark}


\section{Final remarks and questions}
\label{remarks}

It was shown by Athanasiadis, Brady and Watt  \cite{AthanasiadisBradyWatt}
that the noncrossing partition posets $[e,c]$ for Coxeter elements $c$ in real reflection groups are EL-shellable; this was extended to well-generated complex reflection groups by M\"uhle \cite{muhle}.  In particular, the open intervals $(e,c)$ are
homotopy Cohen--Macaulay.  They also have predictable Euler characteristics, that is, M\"obius functions $\mu(e,c)$.

  Analogously, Theorem~\ref{thm:flag-f-vector}  allows one to compute for
regular elliptic elements $c$ in $\GL_n(\FF_q)$ that the interval $[e,c]$ in 
the absolute order on $\GL_n(\FF_q)$ has
\[
\mu(e,c)=\sum_{\alpha=(\alpha_1,\ldots,\alpha_m)} (-1)^{m} f_\alpha[e,c]
=\sum_{\alpha=(\alpha_1,\ldots,\alpha_m)} (-1)^{m} q^{\varepsilon(\alpha)} \cdot (q^n-1)^{m-1}.
\]
We do not suggest any simplifications for this last expression.

\begin{question}
Is the open interval $(e,c)$ in the absolute order 
on $\GL_n(\FF_q)$ homotopy Cohen--Macaulay?  Is it furthermore shellable?
\end{question}

\noindent 
Homotopy Cohen--Macaulayness would imply two weaker conditions:
\begin{compactenum}[(i)]
\item
$(-1)^{\ell(y) - \ell(x)} \mu(x, y) \geq 0$ for all $x \leq y$ in $[e,c]$, and
\item
for $i < n-2$, one has vanishing reduced homology $\tilde{H}_i((e,c),\ZZ)=0$.
\end{compactenum}
Condition (i) is easily seen to hold for $n = 2$ or $n = 3$ and any $q$; 
in addition, we have checked by direct computation that it holds for $n = 4$ if $q = 2$ or $3$.  

Condition (ii) is trivial for $n = 2$.  For $n=3$, it amounts to connectivity
of the bipartite graph which is the Hasse diagram for $(e,c)$, and one
can give a direct proof (using Proposition~\ref{chain-subspace proposition}) 
that this graph is connected.  
For $n=4$ and $q=2$ we have checked in Sage \cite{sage} that $\tilde{H}_i((e,c),\ZZ)=0$ 
for $i=0,1$ and $\tilde{H}_2((e,c),\ZZ)=\ZZ^{|\mu(e,c)|}=\ZZ^{1034}.$

Similarly, it was shown by Athanasiadis and Kallipoliti \cite{AthanasiadisKallipoliti} that,
after removing the bottom element $e$, the absolute order on  all
of $\SS_n$ gives a \emph{constructible} simplicial complex, and hence also this poset is homotopy Cohen--Macaulay.  
In type $B_n$, it is open whether removing the bottom element from
the absolute order gives a homotopy Cohen--Macaulay complex; however,
Kallipoliti \cite{Kallipoliti} showed that when one restricts to
the order ideal which is the union of all intervals below Coxeter elements,
one obtains a homotopy Cohen--Macaulay complex.

\begin{question}
\label{abs-order-homotopy-question}
After removing the bottom element from the absolute order on all of $\GL(V)$, say for $V=\FF_q^n$, does one obtain a homotopy Cohen--Macaulay simplicial complex?  What about the order ideal which is the union of all intervals below Singer cycles?
\end{question}

\noindent
For example, for $\GL_3(\FF_2)$, every maximal element in the absolute order
is already a Singer cycle, so that the two simplicial 
complexes in Question~\ref{abs-order-homotopy-question} are the same.
Both have reduced simplicial homology vanishing in dimensions $0,1$,
and isomorphic to $\ZZ^{838}$ in dimension $2$.

In terms of Sperner theory, the poset $[e,c]$ is rank-symmetric and rank-unimodal by \eqref{interval-rank-sizes}, and is self-dual by Proposition~\ref{prop:self-duality}.  This raises a question, suggested by Kyle Petersen.

\begin{question}
For every Singer cycle $c$ in $\GL_n(\FF_q)$, does the absolute order interval $[e,c]$ have a symmetric chain decomposition?
\end{question}

\noindent
The local self-duality proven in Proposition~\ref{prop:self-duality} 
also implies that, for any $c$ in $\GL_n(\FF_q)$,
the {\it Ehrenborg quasisymmetric function} encoding
the flag $f$-vector of the ranked poset $[e,c]$ 
will actually be a symmetric function; see \cite[Thm.\ 1.4]{Stanley-flag}.  
When $c$ is regular elliptic, Theorem~\ref{thm:flag-f-vector}  
lets one compute this symmetric function explicitly, but we did
not find the results suggestive.

Lastly, we ask how the poset $[e,c]$ in $\GL_n(\FF_q)$ depends upon the choice of Singer cycle $c$.  

\begin{question}
Do all Singer cycles $c$ in $\GL_n(\FF_q)$ have isomorphic posets~$[e,c]$?
\end{question}

Certainly $[e,c]$ and $[e,c']$ are poset-isomorphic whenever $c, c'$ are conjugate, and whenever $c'=c^{-1}$.  However, not all Singer cycles can be related by conjugacy and taking inverses.  A similar issue arises for Coxeter elements $c$ in finite reflection groups $W$.  For real reflection groups, all Coxeter elements {\it are} $W$-conjugate.  For well-generated complex reflection groups, they are all related by what Marin and Michel \cite{MM} call {\it reflection automorphisms}, and these give rise to the desired poset isomorphisms $[e,c] \cong [e,c']$; see Reiner--Ripoll--Stump \cite{RRS}.

\begin{remark}
In spite of Theorem~\ref{thm:flag-f-vector}, within some $\GL_n(\FF_q)$ there exist \emph{regular elliptic} elements $c'$ and Singer cycles $c$ for which $[e,c'] \not\cong [e,c]$.
For example, the Singer cycles in $\GL_4(\FF_2)$ are the elements $c$ with
characteristic polynomial $t^4+t+1$ or $t^4+t^3+1$, while the elements $c'$ having characteristic polynomial $1 + t + t^2 + t^3 + t^4$ are regular elliptic but not Singer cycles;
such $c'$ have multiplicative order $5 \neq 15=2^4-1=|\FF_{2^4}^\times|$.
One can check that $[e,c] \not\cong [e,c']$, for example by computing the 
determinants of the $\{0,1\}$-incidence matrices between ranks $1$ and $3$ 
for the two intervals. 
\end{remark}

\appendix
\section{Proof of Proposition~\ref{character values prop}}
\label{technical-proof}

We recall here the statement of the proposition,
giving certain irreducible character values 
for $\GL_n:=\GL_n(\FF_q)$ on
the element $z_k$ in $\CC[\GL_n]$ given by $z_k = \sum_{g \colon \codim(V^g) = k} g$.

\begin{recap}
One has these normalized character values on $z_k$ for certain $\chi^{U,\lambda}$.
\begin{compactenum}[(i)]
\item
For any primary irreducible $\GL_n$-character $\chi^{U,\lambda}$ with
the cuspidal $U \neq \one$ nontrivial,
\[
\normchi^{U, \lambda}(z_k) = (-1)^k q^{\binom{k}{2}} \qbin{n}{k}{q}.
\]
\item
For $U = \one = \one_{\GL_1}$ and $\lambda = \hook{d}{n}$ a hook, we have
\[
\normchi^{\one, \hook{d}{n}}(z_k) =\P_k(q^{-d})
\]
where $\P_k(x)$ is the following polynomial in $x$ of degree $k$:
\[
\tag{\ref{definition-of-P}}
\P_k(x) := (-1)^k q^{\binom{k}{2}} \left(
\qbin{n}{k}{q} + 
\frac{1 - q^n}{[n - k]!_q}  
\sum_{j = 1}^k \frac{[n - j]!_q}{[k - j]!_q} q^{j(n - k)} x \cdot (x q^{n - j + 1}; q)_{j - 1}
\right).
\]
\end{compactenum}
\end{recap}

\begin{remark}
Taking $d=0$ in Proposition~\ref{character values prop}(ii),
the character $\normchi^{\one,(n)}$ is the trivial 
character $\one_{\GL_n}$.  Hence 
$\normchi^{\one,(n)}(z_k)=r_q(n,k)$ is the $k$th rank size
for the absolute order on $\GL_n$, as 
computed in Proposition~\ref{abs-order-rank-sizes}.  
It is not hard to check
that the formula for $r_q(n,k)$ given there
is consistent with the $d=0$ case of 
Proposition~\ref{character values prop}(ii),
that is, with $\P_k(1)$.
\end{remark}

\begin{proof}[Proof of Proposition~\ref{character values prop}]
We begin the proof of both assertions (i) and (ii)
with a M\"obius function calculation as in the proof of 
Proposition~\ref{abs-order-rank-sizes}.  

Fix a character $\chi$.  Since $\chi$ is a class function, 
one has for any fixed subspace $X$ in $V$ of codimension $k$ that 
\[
\normchi(z_k):= \sum_{\substack{g \in \GL_n\colon \\\codim(V^g) = k}} \normchi(g) = \qbin{n}{k}{q}  F(X)
\quad \text{ where }
F(X) :=
\sum_{\substack{g \in \GL_n\colon \\V^g = X}} \normchi(g).
\]
Rather than $F(X)$, it is more convenient to compute 
\[
G(X) := \sum_{\substack{g \in \GL_n\colon \\V^g \supseteq X}} \normchi(g) 
       =  \sum_{Y \colon X \subseteq Y \subseteq V} F(Y).
\]
Then by M\"obius inversion \cite[Ex.~3.10.2]{EC1} on the 
lattice of subspaces of $\Fq^n$ we have
\[
F(X) = \sum_{Y\colon X \subseteq Y \subseteq V}  
         (-1)^{\dim Y - \dim X} q^{\binom{\dim Y - \dim X}{2}}  G(Y).
\]
It follows that
\begin{equation}
\label{after mobius}
\normchi(z_k) = 
\qbin{n}{k}{q} 
\sum_{j = 0}^k (-1)^{k - j} q^{\binom{k - j}{2}} \qbin{k}{j}{q} G(Y)
\end{equation}
where $Y := Y_j$ is any particular subspace of codimension $j$.
Thus it only remains to compute 
$G(X)$ where $X$ is a particular codimension-$k$ subspace;
for concreteness, we take $X$ to be the span of the 
first $n-k$ standard basis vectors in $V$.

If $k = 0$ then $X = V$ and $G(X)=\normchi(e) = 1$.  
Thus, in what follows we assume $k \geq 1$.

Abbreviate a tower of groups 
\[
\begin{array}{rcccc}
\GL_n &\supset& P& \supset& H \\
     &       & \Vert& &\Vert \\
     &       & 
\left\{ \left[ \begin{matrix} A_1 &B\\
                              0 & A_2
               \end{matrix} \right] \right\}& & 

\left\{ \left[ \begin{matrix} \one_{n-k} &B\\
                              0 & A_2
               \end{matrix} \right] \right\} 
\end{array}
\]
in which $P$ is the parabolic (block upper triangular) subgroup stabilizing
$X$ (\emph{not} necessarily pointwise), and
$H$ is the subgroup of $P$ that fixes $X$ pointwise.
Also recall that inside $P$ 
one finds the block-diagonal product group $\GL_{n-k} \times \GL_k$.
Still fixing a $\GL_n$-character $\chi$, we compute
\[
G(X) = \sum_{h \in H} \normchi(h) 
 = \frac{|H|}{\chi(e)} 
      \left\langle  \quad \chi\downarrow^{\GL_n}_{H} \, , \quad \one_{H} \quad \right\rangle_{H} 
 =  \frac{|H|}{\chi(e)} \left\langle \quad \chi  \, , \quad \one_{H} \uparrow^{\GL_n}_{H}  \quad \right\rangle_{\GL_n}
\]
via Frobenius Reciprocity for induction $(-)\uparrow^{\GL_n}_{H}$ and 
restriction $(-)\downarrow^{\GL_n}_{H}$.  
The map sending
\[
p=\left[ \begin{matrix} A_1 &B\\
                      0 & A_2
       \end{matrix} \right] 
\longmapsto A_1
\]
induces a bijection $P/H \rightarrow \GL_{n-k}$ showing that 
the left-translation action of $p$ on cosets $P/H$ is isomorphic to the left-regular action of $A_1$ on $\GL_{n-k}$.  Hence
\[
\one_{H} \uparrow^{P}_{H} 
   = 
   \left( \CC \GL_{n-k} \otimes \one_{\GL_k} \right)
     \Uparrow_{\GL_{n-k} \times \GL_k}^P,  
\]
where $\CC \GL_{n-k}$ is the regular representation of $\GL_{n-k}$,
and recall that $(-)\Uparrow_{\GL_{n-k} \times \GL_k}^P$ denotes inflation of
a $\GL_{n-k} \times \GL_k$-representation to a $P$-representation
by precomposing with the surjection 
$P \twoheadrightarrow \GL_{n-k} \times \GL_k$.
Hence, via transitivity of induction, one can rewrite
\[
\one_{H} \uparrow^{\GL_n}_{H} 
  = \left( \one_{H} \uparrow^{P}_{H} \right) \uparrow^{\GL_n}_{P}
=   \left( \CC \GL_{n-k} \otimes \one_{\GL_k} \right) 
       \Uparrow_{\GL_{n-k} \times \GL_k}^P
         \uparrow_P^{\GL_n} 
= \CC \GL_{n-k} \hcprod \one_{\GL_{k}}.
\]
To apply \eqref{after mobius}, we need to compute
for $\codim(X)=k \geq 1$ the values
\begin{equation}
\label{regular-rep-sum}
\begin{aligned}
G(X) 
&=
\frac{|H|}{\chi(e)} 
\left\langle \quad \chi  \, , \quad 
\CC \GL_{n - k} *  \one_{\GL_{k}} \quad 
\right\rangle_{\GL_n} \\
&=
\frac{|H|}{\chi(e)} 
\sum_{\llambda} \chi^\llambda(e)
\left\langle \quad \chi  \, , \quad 
\chi^\llambda *  \chi^{\one_{\GL_1},(k)} \quad 
\right\rangle_{\GL_n} 
\end{aligned}
\end{equation}
with $\chi^\llambda$ running through $\Irr(\GL_{n-k})$.
We compute this now for $\chi$ as in assertions (i), (ii).

\subsection*{Assertion (i).}
Here $\chi=\chi^{U,\lambda}$ with $U \neq \one_{\GL_1}$.
In this case, 
$
\left\langle \chi^{U,\lambda}  \, , \quad 
\chi^\llambda *  \chi^{\one_{\GL_1},(k)} 
\right\rangle_{\GL_n}
$
always vanishes, since 
$\chi^\llambda *  \chi^{\one_{\GL_1},(k)}$
cannot have the primary irreducible 
$\chi^{U,\lambda}$ as a constituent:  its irreducible constituents
$\chi^{\underline{\mu}}$ must each have $\mu$ assigning the cuspidal
$\one_{\GL_1}$ a partition of weight at least $k$,
and hence are not irreducibles primary to $U$.
Consequently, \eqref{after mobius} gives the desired answer
\[
\normchi^{U, \lambda}(z_k) 
 = \qbin{n}{k}{q} (-1)^k q^{\binom{k}{2}} \qbin{k}{0}{q} \cdot 1 
 = (-1)^k q^{\binom{k}{2}} \qbin{n}{k}{q}.
\]

\subsection*{Assertion (ii).}
Here $\chi=\chi^{\one_{G_1},\hook{d}{n}}$.
We claim that $k > 0$ and {\it Pieri's rule} \cite[(5.16)]{Macdonald} for expanding
the induction product of $\chi^\lambda$ and $\chi^{(k)}$  imply that
almost every $\chi^{\llambda}$ in $\Irr(\GL_{n-k})$ 
has the inner product 
$
\left\langle \chi^{\one_{\GL_1},\hook{d}{n}}  \, , \quad 
\chi^\llambda *  \chi^{\one_{\GL_1},(k)} 
\right\rangle_{\GL_n}
$
vanishing, unless both
\begin{itemize}
\item $k \leq n - d$, and 
\item $\chi^{\llambda}=\chi^{\one_{\GL_1},\lambda}$
for either $\lambda=\hook{d}{n - k}$ or $(n - k - d + 1, 1^{d - 1})$, 
\end{itemize}
in which case the inner product is $1$.  
Hence, starting with \eqref{regular-rep-sum}, we compute
\begin{align*}
G(X) 
&= \frac{|H|}{ \chi^{\one_{\GL_1},\hook{d}{n}}(e)} 
\sum_{\llambda} \chi^\llambda(e)
\left\langle \quad \chi^{\one_{\GL_1},\hook{d}{n}}  \, , \quad 
\chi^\llambda *  \chi^{\one_{\GL_1},(k)} \quad 
\right\rangle_{\GL_n} \\
&=\frac{|H|}{ \chi^{\one_{\GL_1},\hook{d}{n}}(e)}
\left(  
\chi^{\one_{\GL_1},\hook{d}{n-k}}(e)
+
\chi^{\one_{\GL_1},(n-k-d+1,1^{d-1})}(e)
\right)\\
& = \frac{|H|}{q^{\binom{d + 1}{2}} \qbin{n - 1}{d}{q}} 
 \left( q^{\binom{d + 1}{2}} \qbin{n - k - 1}{d}{q} + q^{\binom{d}{2}} \qbin{n - k - 1}{d - 1}{q}\right) \\
& = \left. (-1)^k (q;q)_k q^{\binom{k}{2}+ k(n - k)-d}
        \qbin{n - k}{d}{q} \middle\slash \qbin{n - 1}{d}{q} \right. .
\end{align*}
Plugging this result into \eqref{after mobius}, after separating out the $j=0$ summand, gives
\begin{align*}
\normchi^{\one, \hook{d}{n}}(z_k)  
& =  (-1)^k q^{\binom{k}{2}}\qbin{n}{k}{q}
 \left(
1 + 
\sum_{j = 1}^{\min(k, n - d)} (-1)^{j} q^{j(n - k) - d}  \qbin{k}{j}{q} 
 \frac{ (-1)^j (q;q)_j  \qbin{n - j}{d}{q}}
            {\qbin{n - 1}{d}{q} }
\right) \\
& = 
 (-1)^k q^{\binom{k}{2}} 
\left(
\qbin{n}{k}{q} + 
\frac{1 - q^n}{[n-k]!_q} 
\sum_{j = 1}^{\min(k, n - d)} q^{j(n - k) - d} 
 (q^{n - j - d + 1}; q)_{j - 1} 
\frac{[n - j]!_q }
{[k -j]!_q  }
\right) \\
& = \P_k(q^{-d}). \qedhere
 \end{align*}
\end{proof}

\bibliography{AbsGL}{}
\bibliographystyle{plain}

\end{document}